\documentclass[brochure,12pt]{bourbaki}
\usepackage[matrix,arrow]{xy}
\usepackage{amssymb,amsfonts,amsmath,footnote}
\usepackage[utf8]{inputenc}
\usepackage[T1]{fontenc}
\usepackage[francais]{babel}
\date{Novembre 2015}
\bbkannee{68\`eme ann\'ee, 2015-2016}
\bbknumero{1105}
\title{Positivité du cotangent logarithmique et conjecture de Shafarevich-Viehweg}
\subtitle{d'apr\`es Campana, P\u{a}un, Taji...}
\author{Beno\^\i t CLAUDON}
\address{Institut \'Elie Cartan de Lorraine\\
Universit\'e de Lorraine\\
UMR 7502 du CNRS\\
B.P. 70239\\
F--54 506 Vand\oe uvre-l\`es-Nancy Cedex}
\email{Benoit.Claudon@univ-lorraine.fr}

\newcommand{\RR}{\mathbb R}

\newcommand{\CC}{\mathbb C}
\newcommand{\QQ}{\mathbb Q}
\newcommand{\ZZ}{\mathbb Z}

\newcommand{\PP}{\mathbb P}

\newcommand{\To}{\longrightarrow}
\newcommand{\abs}[1]{\left\vert#1\right\vert}

\newcommand{\set}[1]{\left\{#1\right\}}

\newcommand{\ram}{\operatorname{Ram}}
\newcommand{\dd}{\mathrm{d}}
\newcommand{\F}{\mathcal{F}}
\newcommand{\E}{\mathcal{E}}
\newcommand{\LL}{\mathcal{L}}
\newcommand{\mob}{\operatorname{Mob}}
\newcommand{\psef}{\operatorname{Psef}}
\newcommand{\rk}{\operatorname{rg}}
\newcommand{\sym}{\operatorname{Sym}}

\newtheorem{affi}[defi]{Affirmation}

\begin{document}
\maketitle

\noindent{\bf INTRODUCTION}

\bigskip
Un des premiers objets intrinsèquement attachés à une variété\footnote{Toutes les variétés considérées dans ce texte seront définies sur $\CC$, lisses et connexes. Sauf mention explicite, elles seront de plus projectives. Nous identifierons également fibrés en droites et (classes d'équivalence linéaire de) diviseurs.} projective lisse $X$ est son fibré canonique : $K_X:=\det(\Omega^1_X)$. Il est de plus bien connu que les propriétés de positivité/négativité de ce fibré en droites gouvernent une grande partie de la géométrie de $X$. Il est alors naturel de se demander si les propriétés en question sont la trace de propriétés vérifiées par le fibré cotangent lui-même. C'est l'objet des récents travaux de Campana et P{\u a}un \cite{CP13,CP14,CP15} qui ont permis d'établir le résultat suivant\footnote{\'Enoncé ainsi, le théorème \ref{th-intro:positivité orbifolde} remonte à l'article \cite{CPet}. Il semble cependant que la démonstration de \emph{loc. cit.} ne soit pas complète : l'utilisation des résultats de Miyaoka (\og Relative deformations of morphisms and applications to fibre
spaces\fg , \emph{Comment. Math. Univ. St. Paul.} \textbf{42} (1993), p. 1–7) n'est pas légitime dans la situation de l'article \cite{CPet}.}.

\begin{theo}\label{th-intro:positivité orbifolde}
Soit $X$ une variété projective lisse avec $K_X$ pseudo-effectif. Le fibré cotangent a alors la propriété suivante : pour tout entier $m\ge1$ et pour tout quotient sans torsion
$${\Omega^1_X}^{\otimes m}\twoheadrightarrow \mathcal{Q},$$
le fibré\footnote{Rappelons que le déterminant d'un faisceau cohérent sans torsion $\E$ est le fibré en droites $(\bigwedge^r\E)^{**}$, avec $r$ le rang de $\E$. Voir également la définition \ref{defi:pente faisceau}.} en droites $\det(\mathcal{Q})$ est pseudo-effectif.
\end{theo}

Il faut bien sûr mettre ce résultat en parallèle avec les travaux de Miyaoka \cite{Mi87} qui a montré que, sous les mêmes hypothèses, le fibré $\Omega^1_X$ était \emph{génériquement semi-positif} : le degré de ses quotients est positif sur toute courbe intersection complète dans le système linéaire d'un diviseur très ample. Le théorème \ref{th-intro:positivité orbifolde} constitue donc une généralisation des travaux de Miyaoka.

La généralisation vient aussi du fait que les articles \cite{CP13,CP15} traitent d'une situation bien plus générale puisqu'ils se placent dans le cadre des paires $(X,\Delta)$ où $X$ est lisse et $\Delta$ un $\QQ$-diviseur dont le support est à croisements normaux (et dont les coefficients sont compris entre 0 et 1). Une telle paire sera appelée une orbifolde dans la suite (conformément à la terminologie de \cite{Ca04,Ca11j}). Sous l'hypothèse de pseudo-effectivité de $K_X+\Delta$, ils montrent qu'un fibré naturellement associé à la paire possède des propriétés similaires à celles énoncées dans le théorème \ref{th-intro:positivité orbifolde}. Une des difficultés vient du fait que ce fibré (qu'ils nomment \emph{fibré cotangent orbifolde}) ne vit pas sur $X$ mais sur un revêtement ramifié de $X$ dont la ramification est en partie contrôlée par le diviseur $\Delta$. Nous renvoyons au paragraphe \ref{sub:cotangent} pour les notions esquissées et pour un énoncé complet (à savoir, le théorème \ref{th:positivité orbifolde}).\\

En plus de la construction du cotangent orbifolde et de la mise au jour de ses propriétés essentielles, une des contributions majeures de la prépublication \cite{CP15} est d'établir un critère d'intégrabilité algébrique pour les feuilletages.
\begin{theo}\label{th-intro:intégrabilité}
Soit $\F$ un feuilletage sur $X$ (supposée projective et lisse) vérifiant $\mu_\alpha^{min}(\F)>0$ pour une classe mobile $\alpha\in\mob(X)$. Le feuilletage est alors algébriquement intégrable : ses feuilles sont ouvertes dans leurs adhérences de Zariski. De plus, ces dernières sont des sous-variétés rationnellement connexes de $X$.
\end{theo}
Les notions de classe mobile et de pentes par rapport à une telle classe sont rappelées dans le paragraphe \ref{sub:pente}. \`A nouveau, ce résultat a des prédécesseurs : les travaux de Bost \cite{JBB} et Bogomolov-McQuillan \cite{BM01} fournissent un critère d'intégrabilité algébrique dans le cas d'une classe intersection complète de diviseurs amples.

Une fois le théorème \ref{th-intro:intégrabilité} établi, la stratégie de la démonstration du théorème est alors transparente. En supposant que la conclusion du théorème \ref{th-intro:positivité orbifolde} soit mise en défaut, nous obtenons une classe mobile $\alpha$ et un sous-faisceau de $T_X^{\otimes m}$ dont la pente par rapport à $\alpha$ est strictement positive. Des résultats de comparaison (\emph{cf.} théorème \ref{th:mumax produit tensoriel}) montre qu'il en est de même pour $T_X$ lui-même : il existe un sous-faisceau de $T_X$ de pente strictement positive (par rapport à $\alpha$). Considérons alors le sous-faisceau de $T_X$ maximal pour l'inclusion parmi les sous-faisceaux de pente positive et notons le $\F$. Par construction, ce faisceau vérifie $\mu_\alpha^{min}(\F)>0$ et un raisonnement standard utilisant des inégalités sur les pentes montre que $\F$ est un feuilletage. Le théorème \ref{th-intro:intégrabilité} montre que $X$ doit être recouverte par des courbes rationnelles mais ceci n'est pas possible puisque nous avons supposé le fibré canonique de $X$ pseudo-effectif (d'après le résultat principal de \cite{BDPP}, ceci revient à dire que $X$ n'est pas uniréglée). L'adaptation de cette stratégie au cas orbifolde nécessitera des aménagements qui feront l'objet de la deuxième partie.\\

Une application spectaculaire de la version orbifolde du théorème \ref{th-intro:positivité orbifolde} consiste en une solution élégante d'une conjecture de Viehweg sur la base des familles de variétés canoniquement polarisées.
\begin{theo}\label{th-intro:conj Viehweg}
Soit $f:X^\circ\to Y^\circ$ un morphisme propre et lisse (à fibres connexes) entre variétés quasi-projectives lisses dont les fibres ont leurs fibrés canoniques amples. Fixons de plus une compactification $Y$ de $Y^\circ$ telle que $D:=Y\setminus Y^\circ$ soit à croisements normaux. Si la variation de $f$ est maximale (c'est-à-dire si l'application de Kodaira-Spencer est injective au point général de $Y^\circ$), le diviseur $K_Y+D$ a alors une dimension de Kodaira maximale et la paire $(Y,D)$ est de log-type général.
\end{theo}
Nous verrons dans la dernière partie de ce texte comment formuler une généralisation de ce résultat qui englobe tout à la fois la conjecture de Viehweg et une conjecture de Campana \cite[conj. 13.29]{Ca11j}, formulation due à Taji \cite{Taj}. Les travaux portant sur les familles de variétés canoniquement polarisées ont pour origine la solution de la conjecture d'hyperbolicité de Shafarevich \cite{Sh63} par Parshin \cite{P68} et Arakelov \cite{A71}. Les travaux de Viehweg et Zuo \cite{VZ00} ont  ensuite ouvert la voie en dimension supérieure comme en témoignent par exemple les articles de Kebekus-Kov{\`a}cs \cite{KK08}, Jabbusch-Kebekus \cite{JKmz} et Patakfalvi \cite{Pat12}. Nous renvoyons également au texte de survol \cite{surveyKeb}.

\subsection*{Remerciements :}
Nous tenons à remercier F. Campana, M. P{\u a}un et B. Taji d'avoir répondu à nos diverses questions. Les discussions avec S. Druel, S.Kebekus et M. Toma ont toutes éclairé l'un ou l'autre des aspects présents dans ce texte, qu'ils en soient ici remerciés. Nous exprimons en particulier notre sincère gratitude à S. Druel pour son aide précieuse apportée au cours de la rédaction de ce manuscrit.

\section{Feuilletages : intégrabilité algébrique et positivité}

\subsection{Digest feuilleté}\label{sub:rappel feuilletage}

Nous commençons par faire un rapide point sur les propriétés usuelles des feuilletages singuliers, propriétés qui nous seront utiles par la suite. Ce paragraphe est largement inspiré de la récente prépublication \cite{D15}.
\begin{defi}\label{def:feuilletage}
Un feuilletage (singulier) sur une variété lisse $X$ est un sous-faisceau $\F\subset T_X$ saturé\footnote{Un sous-faisceau $\mathcal{\F}\subset \mathcal{E}$ est dit saturé dans $\mathcal{E}$ si le quotient $\mathcal{E}/\F$ est sans torsion.} dans $T_X$ et stable par le crochet de Lie. Le feuilletage sera dit régulier si $\F$ est un sous-fibré de $T_X$.\\
Le rang $r:=\rk(\F)$ est par définition le rang générique de $\F$ ; de même, le co-rang de $\F$ est l'entier $q:=\dim(X)-r$.\\
Les feuilles de $\F$ sont par définition les feuilles de $\F_{\vert X^\circ}$ où $X^\circ$ est le plus grand ouvert de $X$ sur lequel $\F$ est régulier. Une feuille $L$ sera dite \emph{algébrique} si elle est ouverte dans sa clôture de Zariski\footnote{La notation $\bar{L}^\mathrm{Zar}$ désigne naturellement l'adhérence de Zariski de $L$ dans $X$.}, ou encore si $\dim(L)=\dim(\bar{L}^{\mathrm{Zar}})$.
\end{defi}

Le faisceau $N_\F:=(T_X/\F)^{**}$ sera appelé le faisceau normal à $\F$ et il est de rang $q$. L'inclusion $N^*_\F\hookrightarrow \Omega^1_X$ fournit une forme $\omega_\F\in H^0(X,\Omega^q_X\otimes \det(N_\F))$ dont le lieu des zéros est de codimension au moins deux. La forme $\omega_\F$ est de plus \emph{localement décomposable} et \emph{intégrable}, c'est-à-dire s'écrit localement (au point général de $X$)
$$\omega_\F=\omega_1\wedge\dots\wedge\omega_q$$
où les 1-formes $\omega_i$ vérifient : $\dd \omega_i\wedge\omega_\F=0$. Réciproquement, la donnée d'une telle forme permet de reconstruire le feuilletage comme le noyau du morphisme de contraction induit par $\omega_\F$ : $T_X\to \Omega^{q-1}_X\otimes \det(N_\F)$.

Le fibré (ou diviseur) canonique de $\F$ sera par définition $K_\F:=\det(\F)$ et nous avons bien entendu $K_\F=K_X+\det(N_\F)$.

Si $\varphi:\hat{X}\to X$ est un morphisme birationnel entre variétés lisses, la donnée d'un feuilletage $\F$ sur $X$ induit un feuilletage $\hat{\F}$ sur $\hat{X}$ : $\hat{\F}$ est le saturé dans $T_{\hat{X}}$ de l'image réciproque de $\F$ par la différentielle de $\varphi$. Leurs fibrés canoniques sont naturellement reliés par la formule suivante : $\varphi_*K_{\hat{\F}}=K_\F$ car la différence $K_{\hat{\F}}-\varphi^*K_\F$ est portée par les diviseurs $\varphi$-exceptionnels.\\

Le concept de feuilletage algébriquement intégrable sera de toute première importance dans la suite de ce texte.
\begin{defi}\label{defi:alg intégrable}
Un feuilletage $\F$ sur $X$ sera dit \emph{algébriquement intégrable} si la feuille passant par un point général de $X$ est algébrique (au sens de la définition \ref{def:feuilletage}).
\end{defi}
Un tel feuilletage s'identifie alors à une fibration\footnote{Une fibration est une application propre, surjective et à fibres connexes.} sur un éclatement de $X$. En effet, $X$ étant supposée projective, la famille des (adhérences de Zariski des) feuilles d'un feuilletage algébriquement intégrable est une sous-variété de la variété de Chow $\mathcal{C}(X)$. Un modèle lisse du graphe d'incidence de cette famille de cycles fournit alors une telle fibration. Il existe donc une modification $\varphi:\hat{X}\to X$ et une fibration $f:\hat{X}\to Z$ (avec $Z$ lisse) telles que $\hat{\F}$ s'identifie au noyau de la différentielle $\dd f$. La suite exacte
$$0\To \hat{\F}\To T_{\hat{X}}\To f^*T_Z$$
montre immédiatement que la différence
$$K_{\hat{X}/Z}-\ram(f)-K_{\hat{\F}}$$
est effective et $f$-exceptionnelle. Nous avons utilisé la définition suivante.
\begin{defi}\label{defi:ramification}
Si $g:U\to V$ est un morphisme entre variétés lisses, le diviseur de ramification est :
$$\ram(g):=\sum_{D\subset V}\left(g^*D-(g^*D)_{red} \right)$$
où, dans la somme ci-dessus, $D$ parcourt l'ensemble des diviseurs premiers de $V$. Si $E$ est un diviseur premier de $U$ et $m$ son coefficient dans $\ram(g)$, nous noterons $m_g(E):=m+1$ l'indice de ramification (ou multiplicité) de $g$ le long de $E$.

L'image (ensembliste) de $\ram(g)$ par $g$ est appelée lieu de branchement.
\end{defi}
De toutes ces remarques, nous tirons :
\begin{prop}\label{prop:canonique feuilletage intégrable}
Dans la situation précédente, si de plus\footnote{Un telle fibration est qualifiée de \emph{nette} dans \cite{Ca04}. Il est possible de construire un tel modèle, voir \cite[lem. 1.3]{Ca04} et également \cite[lem. 7.3]{V83}.} tout diviseur $f$-exceptionnel est également $\varphi$-exceptionnel, le fibré canonique de $\F$ est donné par :
$$K_\F=\varphi_*(K_{\hat{X}/Z}-\ram(f)).$$
\end{prop}

Concluons ce paragraphe par la notion de diviseur invariant/horizontal par un feuilletage.
\begin{defi}\label{def:diviseur invariant}
Un diviseur premier $D\subset X$ sera dit \emph{invariant} par le feuilletage $\F$ si sa restriction $D_{\vert X^\circ}$ est une réunion de feuilles de $\F_{\vert X^\circ}$. Si $\Delta$ est un $\QQ$-diviseur, nous noterons $\Delta^{inv}$ le $\QQ$-diviseur des composantes de $\Delta$ invariantes par $\F$ (et affectées des mêmes multiplicités que dans $\Delta$) et sa partie \emph{horizontale} sera $\Delta^{hor}:=\Delta-\Delta^{inv}$.
\end{defi}
La terminologie \emph{horizontale} est bien évidemment calquée sur la situation d'une fibration.

\subsection{Positivité des feuilletages algébriquement intégrables}

Nous allons constater que, sous des hypothèses assez faibles sur $X$, le fibré canonique (tordu) d'un feuilletage algébriquement intégrable est pseudo-effectif. Comme nous en aurons besoin dans la suite, nous énonçons dès maintenant un résultat impliquant la présence d'un diviseur additionnel. Nous adoptons ici la terminologie de \cite{Ca04,Ca11j} selon laquelle une structure orbifolde sur $X$ consiste en la donnée d'un $\QQ$-diviseur $\Delta$ dont le support est à croisements normaux et dont les coefficients sont compris entre 0 et 1.
\begin{theo}\label{th:positivité feuilletages}
Soit $\F$ un feuilletage algébriquement intégrable sur $X$ et supposons $X$ munie d'une structure orbifolde $\Delta$ telle que $K_X+\Delta$ soit pseudo-effectif. Le fibré canonique $K_\F+\Delta^{hor}$ est alors pseudo-effectif.
\end{theo}
Ce résultat (une reformulation de \cite[th. 3.3]{CP13}) est une généralisation de \cite[lem. 2.14]{Hor}. Nous recommandons également la lecture de \cite[\S 4]{D15} pour des résultats similaires mais s'exprimant en termes des singularités de $\F$.\\

Nous allons esquisser la démonstration du théorème \ref{th:positivité feuilletages} qui est une conséquence de résultats maintenant classiques concernant la positivité des images directes du fibré canonique tordu. L'énoncé suivant est essentiellement contenu dans \cite[th. 4.13]{Ca04}.

\begin{theo}\label{th:campana-fujino}
Soit $(X,\Delta)$ une structure orbifolde sur $X$. Pour toute fibration $f:X\to Z$ entre variétés projectives lisses et tout $m$ assez divisible, le faisceau $f_*\left(\mathcal{O}_X(m(K_{X/Z}+\Delta))\right)$ est faiblement positif.
\end{theo}
Il s'agit bien évidemment d'une généralisation au cas $\Delta\neq0$ des résultats de positivité de Viehweg \cite{V83}. Nous renvoyons d'ailleurs à \emph{loc. cit.} pour la notion de faible positivité puisque nous n'en ferons qu'un bref usage dans ce texte. Signalons d'ailleurs que les techniques permettant d'établir le théorème \ref{th:campana-fujino} remontent à l'article fondateur \cite{V83} mais que l'introduction d'une structure orbifolde en augmente considérablement la portée. Signalons également l'article \cite{Fuj14} dans lequel se trouve une exposition détaillée de ce cercle d'idées.

Nous utiliserons le théorème \ref{th:campana-fujino} à travers le corollaire suivant qui est peu ou prou équivalent à l'énoncé du théorème \ref{th:positivité feuilletages}.

\begin{coro}\label{cor:positivité avec ramification}
Soient $f:X\to Z$ une fibration entre variétés projectives lisses et $\Delta$ une structure orbifolde sur $X$ ayant la propriété suivante : $K_{X_z}+\Delta_{\mid X_z}$ est pseudo-effectif pour $z\in Z$ général. Le diviseur
$$K_{X/Z}+\Delta^{hor}-\ram(f)$$
est alors pseudo-effectif.
\end{coro}

Voyons comment déduire le théorème \ref{th:positivité feuilletages} du corollaire précédent. 
\begin{proof}[Démonstration du théorème \ref{th:positivité feuilletages}]
Soit $(X,\Delta)$ une paire avec $K_X+\Delta$ pseudo-effectif et $\F$ un feuilletage algébriquement intégrable. Considérons $\varphi:\hat{X}\to X$ une modification de $X$ pour laquelle $\hat{\F}$ s'identifie au feuilletage induit par $f:\hat{X}\to Z$. Nous supposerons que les applications vérifient les conditions de la proposition \ref{prop:canonique feuilletage intégrable}. La paire $(X,\Delta)$ étant en particulier log-canonique, nous pouvons supposer que l'application vérifie également :
\begin{equation}\label{eq:adjonction}
K_{\hat{X}}+\hat{\Delta}=\varphi^*(K_X+\Delta)+E
\end{equation}
où $E$ est effectif et $\varphi$-exceptionnel et où $\hat{\Delta}$ est un diviseur orbifolde vérifiant $\varphi_*(\hat{\Delta})=\Delta$. Nous avons alors bien évidemment $\varphi_*(\hat{\Delta}^{hor})=\Delta^{hor}$ (les notions d'horizontalité étant définies par rapport à $\hat{\F}$ et $\F$ respectivement). D'autre part, la formule (\ref{eq:adjonction}) montre que le diviseur $K_{\hat{X}_z}+\hat{\Delta}_{\vert \hat{X}_z}$ est pseudo-effectif pour $z\in Z$ général et nous pouvons appliquer le corollaire \ref{cor:positivité avec ramification} : le diviseur $K_{\hat{X}/Z}+\hat{\Delta}^{hor}-\ram(f)$ est lui-aussi pseudo-effectif. La proposition \ref{prop:canonique feuilletage intégrable} nous permet enfin de conclure que le diviseur
$$K_{\F}+\Delta^{hor}=\varphi_*(K_{\hat{X}/Z}+\hat{\Delta}^{hor}-\ram(f))$$
est bien pseudo-effectif.
\end{proof}

Pour finir, montrons comment obtenir le corollaire \ref{cor:positivité avec ramification} à partir du théorème \ref{th:campana-fujino}. C'est l'objet du lemme suivant qui permet en quelque sorte d'éliminer la ramification par un changement de base. Ce résultat est implicite dans les pages 570-571 de \cite{Ca04} (voir également \cite[lem. 4.2]{D15} et les références qui y sont mentionnées).
\begin{lemm}\label{lem:elimination ramification}
Soit $f:X\to Z$ une fibration entre variétés projectives lisses. Il existe alors un changement de base génériquement fini $\tilde{Z}\to Y$ tel que, pour toute désingularisation $\tilde{X}$ de la composante principale de $\tilde{Z}\times_Z X$, le diagramme
$$\xymatrix{\tilde{X}\ar[r]^\varphi\ar[d]_{\tilde{f}} & X\ar[d]^f \\
\tilde{Z}\ar[r] & Z}$$
vérifie : $\varphi_*(K_{\tilde{X}/\tilde{Z}})=\deg(\tilde{X}/X)(K_{X/Z}-\ram(f))$.
\end{lemm}
\begin{proof}[Esquisse de démonstration]
Quitte à changer de modèle birationnel, nous pouvons supposer que le lieu de branchement de $f$ est un diviseur à croisements normaux. Pour chaque composante $D\subset Z$ du lieu de branchement de $f$, nous considérons alors l'entier suivant (les notations font référence à la définition \ref{defi:ramification}) :
$$k_D:=\mathrm{ppcm}\left(m_f(E)\mid E\subset\ram(f)\,\mathrm{et}\, f(E)=D\right).$$
Nous pouvons donc considérer un revêtement ramifié $\tilde{Z}$ de $Z$ ramifiant à l'ordre $k_D$ au-dessus de $D$ pour toute composante $D$ (\emph{cf.} la proposition \ref{prop:exisence Kawa cover} ci-dessous). Si $\tilde{X}$ désigne une désingularisation comme dans l'énoncé, un calcul en coordonnées locales montre que la différence
$$K_{\tilde{X}/\tilde{Y}}-\varphi^*(K_{X/Y}-\ram(f))$$
est $\varphi$-exceptionnelle. Ceci vient du fait que, au-dessus d'un point général de toute composante de $\ram(f)$, l'application $\varphi$ est étale.
\end{proof}
\begin{proof}[Démonstration du corollaire \ref{cor:positivité avec ramification}]
Considérons un changement de base $\tilde{f}:\tilde{X}\to \tilde{Z}$ fourni par le lemme précédent. Quitte à modifier encore $\tilde{X}$, nous pouvons supposer que le support de $\tilde{\Delta}^{hor}$ est à croisements normaux. De plus, la restriction du fibré $K_{\tilde{X}/\tilde{Z}}+\tilde{\Delta}^{hor}$ à la fibre générale de $\tilde{f}$ reste bien pseudo-effective et, quitte à ajouter un diviseur ample arbitrairement petit, nous pouvons même supposer qu'elle est effective. Cela signifie en particulier que
$$\tilde{f}^*\tilde{f}_*\left(\mathcal{O}_X(m(K_{\tilde{X}/\tilde{Z}}+\tilde{\Delta}^{hor}))\right)\twoheadrightarrow \mathcal{O}_X\left(m(K_{\tilde{X}/\tilde{Z}}+\tilde{\Delta}^{hor})\right)$$
est non nul donc génériquement surjectif (pour $m$ suffisamment divisible).  D'après le théorème \ref{th:campana-fujino}, nous en déduisons que le $\QQ$-diviseur $K_{\tilde{X}/\tilde{Z}}+\tilde{\Delta}^{hor}$ est faiblement positif donc pseudo-effectif. En prenant l'image directe par l'application $\varphi:\tilde{X}\to X$, nous en déduisons que
$$K_{X/Z}+\Delta^{hor}-\ram(f)\sim_\QQ \varphi_*(K_{\tilde{X}/\tilde{Z}}+\tilde{\Delta}^{hor})$$
est également pseudo-effectif.
\end{proof}

\subsection{Pente par rapport à une classe mobile}\label{sub:pente}

Avant d'établir un critère général d'intégrabilité algébrique des feuilletages, nous introduisons les notions nécessaires, à savoir celles de classe mobile et de pente d'un faisceau par rapport à une telle classe.
\begin{defi}\label{defi:classe mobile}
Une classe $\alpha\in N_1(X)_\RR$ est dite \emph{mobile} si elle vérifie $\alpha\cdot D\ge0$ pour tout diviseur effectif $D$. Le cône convexe fermé formées des classes mobiles sera noté $\mob(X)$.
\end{defi}
Par définition, le cône $\mob(X)$ est donc dual du cône $\psef(X)$ des diviseurs pseudo-effectifs (clôture du cône des diviseurs effectifs). La terminologie \emph{mobile} est en grande partie motivée par le résultat principal de \cite{BDPP}.
\begin{theo}\label{th:BDPP}
Si $X$ est une variété projective lisse, le cône $\mob(X)$ est engendré par l'une ou l'autre des familles suivantes :
\begin{enumerate}
\item les classes de courbes $[C_t]$ où $(C_t)_{t\in T}$ est une famille couvrante.
\item les classes de la forme $\varphi_*(H_1\cap\dots\cap H_{n-1})$ où les $H_i$ sont des diviseurs amples sur l'espace total du morphisme birationnel $\varphi:\hat{X}\to X$ et $n=\dim(X)$.
\end{enumerate}
\end{theo}

Les classes mobiles constituent un cadre naturel pour les notions de stabilité des faisceaux. L'introduction de ce cadre remonte à l'article \cite{CPet}. Nous renvoyons également à \cite{GKP} pour un panorama complet, traitant en sus le cas où l'espace ambiant est normal. Rappelons qu'à tout faisceau cohérent sans torsion $\E$, nous avons associé son déterminant $\det(\E):=\left(\bigwedge^r\E\right)^{**}$ (avec $r$ le rang de $\E$). Cette construction s'étend naturellement aux faisceaux cohérents (voir \cite[chap.5, \S 6]{livreKob}) et est additive dans les suites exactes : si
$$0\To \E\To \F\To \mathcal{G}\To0$$
est une suite exacte de faisceaux cohérents, leurs déterminants vérifient
$$\det(\F)=\det(\E)+\det(\mathcal{G}).$$
\begin{defi}\label{defi:pente faisceau}
Si $\alpha\in\mob(X)$ est une classe mobile, la \emph{pente} d'un faisceau cohérent $\E$ de rang $r>0$ par rapport à $\alpha$ est donnée par
$$\mu_\alpha(\E):=\frac{1}{r}\det(\E)\cdot\alpha.$$
\end{defi}
Une des raisons pour lesquelles le cadre des classes mobiles semble bien adapté pour les questions de pentes et de stabilité vient du fait que la pente dépend \emph{linéairement} de $\alpha$ alors que, dans le cas \og classique\fg ~où $\alpha=H^{n-1}$, la dépendance en $H$ est plus compliquée. Il faut également noter que les classes mobiles bénéficient d'une souplesse birationnelle : l'image directe d'un diviseur effectif par une application génériquement fini étant effective, cela implique que l'image réciproque d'une classe mobile par une telle application est encore mobile.

Les propriétés suivantes sont standard (dans cet énoncé et dans la suite, les notations $\hat{\otimes}$ et $\hat{\sym}$ font référence au quotient par la torsion\footnote{Le produit tensoriel de deux faisceaux sans torsion ne l'est pas nécessairement.} des opérations usuelles de produit tensoriel et puissance symétrique).
\begin{prop}\label{prop:propriété de la pente}
Si $\alpha$ est une classe mobile et les faisceaux considérés sans torsion, nous avons alors :
\begin{enumerate}
\item $\mu_\alpha(\E^*)=-\mu_\alpha(\E)$.
\item si $\F\subset$ est saturé dans $\E$, alors
$$\mu_\alpha(\E)=\frac{\rk(\F)}{\rk(\E)}\mu_\alpha(\F)+\frac{\rk(\E)-\rk(\F)}{\rk(\E)}\mu_\alpha(\E/\F).$$
\item $\mu_\alpha(\E\hat{\otimes}\F)=\mu_\alpha(\E)+\mu_\alpha(\F)$.
\item $\mu_\alpha(\hat{\sym}^m\E)=m\mu_\alpha(\E)$ pour tout $m\ge1$.
\item $\mu_\alpha(\bigwedge^p\E/\mathrm{Tor})=p\mu_\alpha(\E)$ pour tout $1\le p\le \rk(\E)$.
\end{enumerate}
\end{prop}

Nous aurons également besoin de considérer les quantités suivantes.
\begin{defi}\label{defi:mumax}
Si $\E$ est un faisceau sans torsion et $\alpha$ une classe mobile, sa pente maximale est par définition le maximum des pentes de ses sous-faisceaux :
$$\mu_\alpha^{max}(\E):=\sup\left\{\mu_\alpha(\F)\mid \F\subset\E\right\}.$$
De même, nous noterons
$$\mu_\alpha^{min}(\E):=\inf\left\{\mu_\alpha(\mathcal{Q})\mid \E\twoheadrightarrow\mathcal{Q}\right\}$$
la pente minimale des quotients\footnote{Comme un faisceau de torsion a un déterminant effectif (voir par exemple \cite[prop. 6.14]{livreKob}), on obtient la même quantité en ne considérant que les quotients sans torsion.} de $E$.
\end{defi}
\begin{rema}\label{rem:dualité mumin-mumax}
Les deux quantités précédentes sont duales l'une de l'autre :
$$\mu_\alpha^{max}(\E)=-\mu_\alpha^{min}(\E^*).$$
De même, il est aisé de vérifier que si $\F\subset \E$ est saturé dans $\E$ et si $\mu_\alpha(\F)=\mu_\alpha^{max}(\E)$ alors l'inégalité suivante est satisfaite :
$$\mu_\alpha^{max}(\E/\F)\le\mu_\alpha(\F)\le \mu_\alpha^{min}(\F).$$
\end{rema}
\begin{rema}\label{rem:pente saturation}
Si $\F\subset \mathcal{E}$ sont deux faisceaux de même rang, la différence $\det(\mathcal{E})-\det(\F)$ est effective et cela entraîne $\mu_\alpha(\mathcal{E})\ge \mu_\alpha(\F)$. Plus généralement, les pentes minimales/maximales suivent la même inégalité :
$$\mu^{min}_\alpha(\mathcal{E})\ge \mu^{min}_\alpha(\F)\quad\mathrm{et}\quad\mu^{max}_\alpha(\mathcal{E})\ge \mu^{max}_\alpha(\F).$$
Pour les pentes maximales, l'inégalité est évidente et, si $\mathcal{Q}$ est un quotient de $\mathcal{E}$, l'image de $\F$ dans $\mathcal{Q}$ est un sous-faisceau $\mathcal{Q}'\subset \mathcal{Q}$ de même rang et nous pouvons appliquer l'inégalité précédente.
\end{rema}
Nous utiliserons ces considérations de pentes principalement \emph{via} leurs conséquences sur les espaces de sections.
\begin{prop}\label{prop:annulation section}
Soient $\E$ et $\F$ deux faisceaux cohérents sans torsion. Si l'inégalité
$$\mu_\alpha^{min}(\E)>\mu_\alpha^{max}(\F)$$
est vérifiée, nous avons alors :
$$\mathrm{Hom}(\E,\F)=0.$$
En particulier, les sections globales  de $\F$ s'annulent
$$H^0(X,\F)=0$$
dès que $\mu_\alpha^{max}(\F)<0$.
\end{prop}

Notons que les notations (et appellations) pour les pentes minimales/maximales sont justifiées puisque ses bornes sont atteintes comme le montre la proposition suivante \cite[prop. 1.3]{CPet} (voir également \cite[cor. 2.25]{GKP}).
\begin{prop}\label{prop:existence déstabilisant max}
Si $\E$ est un faisceau cohérent et si $\alpha$ est une classe mobile, il existe alors un unique sous-faisceau $\F\subset\E$ maximal pour l'inclusion tel que :
$$\mu_\alpha(\F)=\mu_\alpha^{max}(\E).$$
En particulier, $\F$ est saturé dans $\E$. Ce sous faisceau sera appelé le \emph{déstabilisant maximal} de $\E$ (pour la classe $\alpha$).
\end{prop}
Le résultat suivant sera crucial dans la suite pour pouvoir entre autre passer de $\Omega^1_X$ à ses puissances tensorielles (une exposition détaillée se trouve dans \cite[\S 4]{GKP}). Il repose \emph{in fine} sur un argument analytique (existence de métrique Hermite-Einstein sur les fibrés stables) qui est dû à M. Toma (voir l'appendice de \cite{CPet}).
\begin{theo}\label{th:mumax produit tensoriel}
Soient $\E$ et $\F$ deux faisceaux cohérents sans torsion et $\alpha$ une classe mobile sur $X$. La pente maximale du produit tensoriel est donnée par :
$$\mu_\alpha^{max}(\E\hat{\otimes}\F)=\mu_\alpha^{max}(\E)+\mu_\alpha^{max}(\F).$$
\end{theo}
\begin{proof}[Esquisse de démonstration]
Des manipulations générales montrent que l'énoncé ci-dessus est essentiellement équivalent au fait de montrer que le produit tensoriel de deux fibrés (faisceaux localement libres) $\alpha$-stables est $\alpha$-semi-stable. De plus, si la classe $\alpha$ est située sur le bord du cône $\mob(X)$, il est possible d'approcher $\alpha$ par des classes situées à l'intérieur de $\mob(X)$ pour lesquelles $\E$ et $\F$ sont encore stables. Il suffit donc de traiter ce dernier cas : $\E$ et $\F$ sont localement libres et $\alpha$ est une classe \emph{positive}. Dans ce cas, la classe $\alpha\in\mathrm{H}^{1,1}(X,\CC)$ s'écrit
$$\alpha=[\omega]^{n-1}$$
avec $\omega$ une métrique hermitienne sur $X$ vérifiant $\partial\bar{\partial}\omega^{n-1}=0$ (une telle métrique est appelée métrique de Gauduchon). Les fibrés $\E$ et $\F$ étant $\omega$-stables, ils admettent des métriques de Hermite-Einstein : c'est la correspondance de Kobayashi-Hitchin dans le cas non kählérien établie par Li et Yau dans \cite{LY}. Leur produit tensoriel admet donc une métrique de Hermite-Einstein et est donc également semi-stable.
\end{proof}

\subsection{Un critère d'intégrabilité algébrique}

Dans cette section, nous énonçons un critère d'intégrabilité algébrique pour les feuilletages qui est l'un des apports essentiels de \cite{CP15}.
\begin{theo}\label{th:critère mobile}
Soient $X$ une variété projective lisse et $\F\subset T_X$ un feuilletage. S'il existe $\alpha\in\mob(X)$ une classe mobile vérifiant $\mu_\alpha^{min}(\F)>0$, le feuilletage $\F$ est alors algébriquement intégrable et ses feuilles sont rationnellement connexes.
\end{theo}
Cet énoncé est dû à J.-B. Bost \cite{JBB} et à F. Bogomolov et M. McQuillan \cite{BM01} dans le cas où la classe $\alpha$ est la classe d'une courbe entièrement contenue dans le lieu régulier de $\F$ (et dans ce cas, la conclusion consiste à dire que les feuilles passant par les points de cette courbe sont algébriques et rationnellement connexes). On pourra consulter également \cite{KSCT}.

Nous reportons les applications directes de ce critère à la section \ref{sub:conséquences} et passons aux grandes lignes de la démonstration du théorème \ref{th:critère mobile}. Les paragraphes suivants sont dédiés respectivement à l'intégrabilité algébrique de $\F$ et à la connexité rationnelle des feuilles de ce dernier.

\subsubsection{Intégrabilité algébrique de $\F$}

Il est remarquable de constater que la démonstration de l'intégrabilité algébrique est une adaptation des arguments des articles sus-mentionnés et remontent, en essence, aux travaux de Hartshorne \cite[th. 6.7]{H68}. Pour simplifier la discussion, nous allons supposer que $\F$ est un feuilletage régulier, c'est-à-dire que $\F$ est un sous-fibré de $T_X$. En un point $x\in X$, nous pouvons donc considérer la feuille qui passe par $x$ : nous noterons $\F_x$ ce germe de sous-variété (analytique). De plus, si nous introduisons
$$\Lambda:=\set{(x,z)\in X\times X\mid z\in\F_x}\quad\mathrm{et}\quad V:=\bar{\Lambda}^{\mathrm{Zar}},$$
l'intégrabilité algébrique de $\F$ est alors équivalente à :
$$\dim(V)=\dim(X)+\rk(\F)=n+r.$$
Comme la dimension de $V$ est évidemment minorée par $n+r$, nous devons établir l'affirmation suivante.
\begin{affi}\label{affi:RiemannRoch}
Pour tout fibré en droites ample $L$ sur $X\times X$, il existe une constante $C:=C(L)>0$ telle que
$$\dim(H^0(V,L_{\mid V}^k))\le Ck^{n+r}$$
pour tout entier $k\ge0$.
\end{affi}
Comme nous ne disposons que d'informations sur $X$ et son plongement dans la variété analytique $\Lambda$, il est alors naturel de filtrer les sections de $L^k$ par l'ordre d'annulation le long de $X$. Si $\mathcal{I}_X$ désigne l'idéal définissant $X$ dans $\Lambda$ et si $m\ge 0$ est un entier, nous avons :
$$0\To L_{\vert \Lambda}^k\otimes\mathcal{I}_X^{m+1}\To L_{\vert \Lambda}^k\otimes\mathcal{I}_X^{m}\To L_{\vert X}^k\otimes\mathcal{I}_X^m/\mathcal{I}_X^{m+1}\simeq L_{\vert X}^k\otimes\sym^m(N^*_{X\mid\Lambda})\To 0.$$
Le fibré normal de $X$ dans $\Lambda$ s'identifie évidemment à $\F$ et nous obtenons la majoration suivante :
\begin{equation}\label{eq:sections formelles}
h^0(V,L^k)\le h^0(\Lambda,L^k)\le \sum_{m\ge0}h^0\left(X,L^k\otimes\sym^m(\F^*)\right).
\end{equation}
La première inégalité est bien entendu une conséquence du fait que $V$ est l'adhérence de Zariski de $\Lambda$. En combinant l'affirmation \ref{affi:RiemannRoch} avec l'inégalité (\ref{eq:sections formelles}) ci-dessus, nous constatons que l'intégrabilité algébrique de $\F$ est alors une conséquence de la proposition suivante.
\begin{prop}\label{prop:estimations sections}
Il existe une constante $C>0$ telle que
$$\sum_{m\ge0}h^0\left(X,L^k\otimes\sym^m(\F^*)\right)\le Ck^{n+r}$$
pour tout $k\ge0$.
\end{prop}
\begin{proof}
Remarquons tout d'abord que la somme est en fait finie. En effet, nous pouvons calculer la pente (par rapport à $\alpha$) du fibré en question :
$$\mu^{max}_\alpha(L^k\otimes\sym^m(\F^*))=kL\cdot\alpha-m\mu^{min}_\alpha(\F)$$
(d'après le théorème \ref{th:mumax produit tensoriel} et la remarque \ref{rem:dualité mumin-mumax}).
Si $m>\dfrac{kL\cdot\alpha}{\mu^{min}_\alpha(\F)}$, la proposition \ref{prop:annulation section} donne l'annulation
$$h^0\left(X,L^k\otimes\sym^m(\F^*)\right)=0.$$
Il nous reste à contrôler les sections pour $m\le\dfrac{kL\cdot\alpha}{\mu^{min}_\alpha(\F)}$. Introduisons pour ce faire $\PP:=\PP(\F)$ la variété des droites projectives du fibré $\F$ et notons $p:\PP\to X$ la projection naturelle. Le fibré tautologique $\mathcal{O}_\PP(1)$ a la propriété suivante :
$$H^0\left(X,L^k\otimes\sym^m(\F^*)\right)=H^0\left(\PP,(p^*L)^k\otimes\mathcal{O}_\PP(m)\right).$$
Il suffit alors de considérer un diviseur ample $A$ sur $\PP$ vérifiant $A\ge p^*L$ ainsi que $A\ge \mathcal{O}_\PP(1)$. 
Si $N=\lceil\dfrac{L\cdot\alpha}{\mu^{min}_\alpha(\F)}\rceil$, nous avons alors
$$h^0\left(X,L^k\otimes\sym^m(\F^*)\right)= h^0\left(\PP,(p^*L)^k\otimes\mathcal{O}_\PP(m)\right)\le h^0\left(\PP,k(N+1)A\right)\le Ck^{n+r-1}$$
pour tout $m\le Nk$. Comme le choix de $A$ et la constante $N$ ne dépendent que de $L$ et $\F$ (et de $\alpha$), la constante $C$ ne dépend pas de $k$. En sommant cette dernière inégalité, nous obtenons
$$\sum_{m\ge0}h^0\left(X,L^k\otimes\sym^m(\F^*)\right)\le NCk^{n+r},$$
ce qui est bien l'estimation attendue.
\end{proof}

\begin{rema}\label{rem:cas F singulier}
La démonstration dans le cas général est similaire mais nécessite de se placer sur l'ouvert $X^\circ$ sur lequel $\F$ est régulier. En effet, la variété $\Lambda$ n'est définie que pour $x\in X^\circ$ et il faut donc filtrer les sections de $L^k$ par l'ordre d'annulation le long de $X^\circ$. Cela ne pose en réalité aucun problème puisque, les faisceaux considérés étant réflexifs, leurs sections globales sont déterminées par leur restriction à $X^\circ$. Nous renvoyons à \cite[\S 4.1]{CP15} pour les détails.
\end{rema}

\subsubsection{Connexité rationnelle des feuilles de $\F$}
Pour compléter la démonstration du théorème \ref{th:critère mobile}, il nous reste à établir la connexité rationnelle de feuilles de $\F$. L'approche proposée dans \cite{CP15} est nouvelle et en fournit une démonstration élégante.

Considérons pour ce faire un modèle $\varphi:\hat{X}\to X$ sur lequel le feuilletage $\hat{\F}$ s'identifie au tangent relatif d'une fibration $f:\hat{X}\to Y$ (et nous supposerons comme d'habitude que les diviseurs $f$-exceptionnels sont aussi $\varphi$-exceptionnels). Raisonnons par l'absurde et supposons que les fibres générales de $f$ ne sont pas rationnellement connexes. Nous pouvons alors considérer le quotient rationnel relatif de $f$ :
$$\xymatrix{X&\hat{X}\ar[l]_\varphi\ar[rr]^{r_{X/Y}}\ar[rd]_f && Z\ar[ld]^g\\
&&Y.&
}$$
Remarquons ici que, quitte à remplacer $\hat{X}$ par un autre modèle birationnel, nous pouvons supposer que $Z$ est lisse et que $r:=r_{X/Y}$ est un morphisme. Les fibres $Z_y$ pour $y$ général dans $Y$ s'identifiant à la base du quotient rationnel de $X_y$, elles ont un fibré canonique pseudo-effectif (conséquence des résultats principaux de \cite{GHS} et \cite{BDPP}). Nous pouvons alors appliquer le corollaire \ref{cor:positivité avec ramification} : $K_{Y/Z}-\ram(g)$ est pseudo-effectif. Or, nous savons que le feuilletage $\mathcal{G}:=T_{Z/Y}$ a pour fibré canonique (\emph{cf.} le paragraphe \ref{sub:rappel feuilletage}) :
$$K_\mathcal{G}=K_{Y/Z}-\ram(g)+E$$
avec $E$ un diviseur $g$-exceptionnel. En particulier, $r^*K_\mathcal{G}$ s'écrit comme la somme d'un diviseur pseudo-effectif et d'un diviseur $f$-exceptionnel, donc $\varphi$-exceptionnel. La contradiction vient alors du fait que $r^*(\mathcal{G})$  s'identifie à un quotient de $\hat{\F}$ (\emph{via} la différentielle de $r$) et des valeurs des pentes : $\mu^{min}_{\varphi^*\alpha}(\hat{\F})=\mu^{min}_\alpha(\F)>0$ et
$$\mu_{\varphi^*\alpha}(r^*\mathcal{G})=-\frac{r^*K_\mathcal{G}\cdot \varphi^*\alpha}{\rk(\mathcal{G})}\le0.$$

\subsection{Conséquences du critère d'intégrabilité algébrique}\label{sub:conséquences}

Le théorème \ref{th:critère mobile} permet de caractériser le fait d'être uniréglé par l'existence de feuilletages particuliers (il s'agit d'une reformulation de \cite{BDPP}).
\begin{coro}\label{cor:caractérisation uniréglée}
Une variété projective lisse $X$ est uniréglée si et seulement si il existe $\alpha\in\mob(X)$ telle que $\mu^{max}_\alpha(T_X)>0$.
\end{coro}
\begin{proof}
Si $\mu^{max}_\alpha(T_X)>0$ (pour une certaine classe mobile $\alpha$), considérons alors $\F$ le déstabilisant maximal de $T_X$ pour $\alpha$. D'après la remarque \ref{rem:dualité mumin-mumax} et le théorème \ref{th:mumax produit tensoriel}, nous avons alors :
$$\mu^{min}_\alpha(\bigwedge^2\F)=2\mu^{min}_\alpha(\F)>\mu_\alpha(\F)\ge \mu^{max}_\alpha(T_X/\F).$$
Le crochet de Lie $\bigwedge^2\F\to T_X/\F$ est donc nul d'après la proposition \ref{prop:annulation section} et $\F$ est bien un feuilletage à pente minimale (par rapport à $\alpha$) strictement positive. Le théorème \ref{th:critère mobile} s'applique et montre que $X$ est uniréglée.

Pour la réciproque, il suffit de constater que, si $X$ est uniréglée, alors $K_X$ n'est pas pseudo-effectif et il existe donc $\alpha\in\mob(X)$ telle que
$$\mu_\alpha(T_X)=-\frac{1}{\dim(X)}K_X\cdot\alpha>0.$$
\end{proof}

Ce critère d'intégrabilité algébrique fournit également des renseignements sur les feuilletages purement transcendants. Un feuilletage $\F$ sur $X$ est dit \emph{purement transcendant} s'il ne passe aucune sous-variété algébrique (de dimension strictement positive) tangente à $\F$ par le point général de $X$.
\begin{coro}\label{cor:feuilletage transcendant}
Le fibré canonique $K_\F$ d'un feuilletage purement transcendant sur $X$ est pseudo-effectif.
\end{coro}
\begin{proof}
Dans le cas contraire, la pente de $\F$ par rapport à une classe mobile $\alpha$ doit être strictement positive. Le déstabilisant maximal $\mathcal{G}$ de $\F$ par rapport à cette classe doit alors être un feuilletage : pour une raison de pente, le morphisme
$$\bigwedge^2\mathcal{G}\To \F/\mathcal{G}$$
induit par le crochet de Lie doit être nul. Le théorème \ref{th:critère mobile} montre alors que $\mathcal{G}$ est un feuilletage algébriquement intégrable et les feuilles de $\mathcal{G}$ fournissent des sous-variétés tangentes à $\F$ ce qui constitue une contradiction.
\end{proof}

En guise de conclusion à cette première partie, nous énonçons le résultat correspondant au théorème principal dans le cas $\Delta=0$.
\begin{theo}\label{th:quotient feuilletage K-psef}
Soient $X$ une variété projective lisse, $\F$ un feuilletage avec $K_\F$ pseudo-effectif et  $m\ge1$ un entier.  Tout quotient de $(\F^*)^{\otimes m}$ a alors un déterminant pseudo-effectif.\\
En particulier, si $K_X$ est pseudo-effectif, tout quotient de $\left(\Omega^1_X\right)^{\otimes m}$ a un déterminant pseudo-effectif.
\end{theo}
\begin{proof}
Supposons par l'absurde que $\mu^{max}_\alpha(\F^{\otimes m})=m\mu^{max}_\alpha(\F)>0$. Le déstabilisant maximal $\mathcal{G}$ de $\F$ pour la classe $\alpha$ est alors un feuilletage (comme dans la démonstration du corollaire \ref{cor:feuilletage transcendant}) de pente minimale positive : $\mathcal{G}$ est donc algébriquement intégrable d'après le théorème \ref{th:critère mobile}. Quitte à modifier $X$, nous supposerons que $\mathcal{G}$ est donné par une fibration $f:X\to Z$. Comme $\mathcal{G}=T_{X/Z}$ est tangent à $\F$, le feuilletage $\F$ descend\footnote{Remarquons que l'intégrabilité est essentielle pour pouvoir affirmer que $\F$ provient de $Z$.} en un feuilletage $\mathcal{H}$ sur $Z$ (voir par exemple \cite[lem. 2.4]{LPT}). Soient alors $z\in Z$ un point général de $Z$ et $V$ un voisinage de $z$ sur lequel $\mathcal{H}$ est régulier (et donc le faisceau correspondant trivial). Si $F$ désigne la fibre générale de $f$ au-dessus de $z$ et $U=f^{-1}(V)$ un voisinage de celle-ci, nous avons alors :
$$0\To \mathcal{G}_{\vert U}\To \F_{\vert U}\To \mathcal{O}_{U}^{\oplus q}\To 0$$
pour un certain entier $q$. En particulier, ceci montre que $(K_\F)_{\vert F}\simeq (K_{\mathcal{G}})_{\vert F}\simeq K_F$ mais ceci est impossible : $(K_\F)_{\vert F}$ est pseudo-effectif alors que $K_F$ ne l'est pas car $F$ est rationnellement connexe d'après le théorème \ref{th:critère mobile}.
\end{proof}
Ce théorème montre en particulier que, si $K_X$ est pseudo-effectif, tout feuilletage (algébriquement intégrable ou non) a un fibré canonique pseudo-effectif : la conclusion du théorème \ref{th:positivité feuilletages} est donc encore valable sans l'hypothèse d'intégrabilité algébrique (dans le cas $\Delta=0$). Un des objectifs de la partie suivante va consister à généraliser ce phénomène au cas $\Delta\neq 0$.

\section{Fibré (co)tangent adapté à une structure orbifolde}\label{sec:orbitangent}

Nous allons introduire ici un \og fibré (co)tangent\fg ~associé à une paire $(X,\Delta)$. Les guillemets reflètent le fait que ce fibré n'est pas défini sur $X$ mais sur un revêtement de $X$ vérifiant certaines conditions de compatibilité avec le diviseur $\Delta$. Parallèlement au théorème \ref{th:quotient feuilletage K-psef}, le fibré cotangent orbifolde va hériter des propriétés de positivité du $\QQ$-diviseur $K_X+\Delta$ : c'est le contenu du théorème \ref{th:positivité orbifolde} ci-dessous.

\subsection{Revêtement $\Delta$-adapté}\label{ssec:rev adapté}

Nous commençons par introduire quelques notations qui seront omniprésentes dans cette partie. Rappelons qu'une \emph{orbifolde} est la donnée d'une paire log-lisse $(X,\Delta)$ avec $X$ lisse, le support de $\Delta$ à croisements normaux et les coefficients de $\Delta$ étant des rationnels compris entre 0 et 1. Un tel diviseur s'écrit donc de manière unique :
$$\Delta=\sum_{i\in I}(1-\frac{b_i}{a_i})\Delta_i$$
avec les conventions suivantes :
\begin{itemize}
\item si $1-\dfrac{b_i}{a_i}<1$, les entiers $a_i$ et $b_i$ sont premiers entre eux et vérifient $0<b_i<a_i$.
\item si $1-\dfrac{b_i}{a_i}=1$, nous posons $a_i=1$ et $b_i=0$.
\end{itemize}
La terminologie \emph{orbifolde} est en partie motivée par le fait que les objets que nous allons définir \og vivent\fg ~sur des revêtements adaptés au diviseur $\Delta$. 
\begin{defi}\label{defi:rev adapté}
Un revêtement adapté à la paire $(X,\Delta)$ (ou encore $\Delta$-adapté) est la donnée d'un revêtement ramifié (plat) et galoisien $\pi:Y\to X$ vérifiant les conditions suivantes :
\begin{enumerate}
\item $Y$ est une variété projective lisse.
\item $\pi$ ramifie exactement à l'ordre $a_i$ au-dessus de $\Delta_i$ : on a une décomposition $\pi^*(\Delta_i)=\sum_{j\in J(i)}a_iD^{(i)}_j$.
\item le support du diviseur $\pi^*(\Delta)+\ram(\pi)$ est à croisements normaux, ainsi que celui du lieu de branchement.
\end{enumerate}
\end{defi}
\begin{rema}\label{rem:def rev adapté}
Nous pouvons remarquer que seuls les entiers $a_i$ jouent un rôle dans la définition ci-dessus. De plus, remarquons que la condition 2 dans la définition ci-dessus signifie que $\pi$ ne ramifie pas au-dessus des composantes ayant coefficient 1 dans $\Delta$.
\end{rema}
\noindent Il est bien connu que de tels revêtements existent, voir par exemple \cite[prop. 4.1.12]{Laz}.
\begin{prop}\label{prop:exisence Kawa cover}
Toute orbifolde $(X,\Delta)$ possède des revêtements $\Delta$-adaptés.
\end{prop}
Ces revêtements ont une description locale extrêmement simple qui aura l'avantage de rendre les calculs aussi transparents que possible.
\begin{prop}\label{prop:formes normales rev}
Soient $\pi:Y\to X$ un revêtement $\Delta$-adapté et $y\in Y$ un point quelconque. Il existe alors un voisinage $y\in U$ invariant par $G_y$ (le groupe d'isotropie de $y$) et des coordonnées $(w_1,\dots, w_n)$ centrées en $y$ (\emph{resp.} $(z_1,\dots, z_n)$ centrées en $\pi(y)$) tels que l'application $\pi$ ait la description suivante :
$$\pi(w_1\dots,w_n)=(w_1^{a_1},\dots,w_k^{a_k},w_{k+1},\dots,w_{n-j},w_{n-j+1}^{m_j},\dots,w_n^{m_1}).$$
Dans l'écriture ci-dessus, nous avons bien entendu supposé que $\pi(\{w_i=0\})\subset \Delta_i$ pour tout $i=1\dots k$ et que $\cup_{\ell=1\dots j}\{w_{n-\ell+1}=0\}\subset \ram(\pi)\setminus \pi^{-1}(\Delta)$.
\end{prop}

\subsection{Fibré cotangent orbifolde associé à un revêtement}\label{sub:cotangent}

Nous commençons par définir le fibré cotangent orbifolde : celui-ci a pour mission de donner un sens aux symboles formels $\dfrac{\dd z_i}{z_i^{(1-b_i/a_i)}}$. Pour cela, nous adoptons le point de vue de Miyaoka \cite[p. 412]{Mi08}. Considérons alors $\pi:Y\to X$ un revêtement $\Delta$-adapté et introduisons l'application de résidu sur la variété $X$ :
$$\Omega^1_X(\log(\lceil\Delta\rceil))\stackrel{\mathrm{res}}{\longrightarrow} \bigoplus_{i\in I}\mathcal{O}_{\Delta_i}\to 0.$$
Nous pouvons prendre son image réciproque par $\pi^*$ :
$$\pi^*\Omega^1_X(\log(\lceil\Delta\rceil))\stackrel{\pi^*\mathrm{res}}{\longrightarrow} \bigoplus_{i\in I}\mathcal{O}_{\pi^*(\Delta_i)}\to 0.$$
Remarquons maintenant que, pour tout $i\in I$, l'expression $D_Y^{(i)}:=\dfrac{1}{a_i}\pi^*(\Delta_i)$ définit un diviseur entier qui vérifie de plus $b_iD_Y^{(i)}\le \pi^*(\Delta_i)$. En particulier, l'application de passage au quotient
$$\mathcal{O}_{\pi^*(\Delta_i)}\twoheadrightarrow\mathcal{O}_{b_iD_Y^{(i)}}$$
est licite. 
\begin{defi}\label{defi:orbi cotangent}
Le faisceau cotangent orbifolde associé à $\pi$ est par définition le sous-faisceau de $\pi^*\Omega^1_X(\log(\lceil\Delta\rceil))$ rendant exacte la suite :
\begin{equation}\label{eq:defi cotangent}
0\To \Omega^1(\pi,\Delta)\To \pi^*\Omega^1_X(\log(\lceil\Delta\rceil))\stackrel{\pi^*\mathrm{res}}{\To} \underset{i\in I}{\bigoplus}\mathcal{O}_{b_iD^{(i)}_Y}\To 0.
\end{equation}
Ce faisceau est donc clairement stable sous l'action de $G=\mathrm{Gal}(\pi)$.
\end{defi}
Ce faisceau est noté $\pi^*\Omega^1(X,\Delta)$ dans \cite{CP13,CP15}. Nous avons opté pour la notation ci-dessus pour insister sur le fait que celui-ci n'existe que sur l'espace total de $\pi$. De plus, comme $\pi$ ne dépend de $\Delta$ qu'à travers les $a_i$, $\Delta$ réapparaît dans la notation pour tenir compte des $b_i$ (ordre d'annulation des sections du cotangent orbifolde le long de $\pi^*(\Delta_i)$).

\begin{rema}\label{rem:inclusion}
Par définition, nous avons donc une chaîne d'inclusions :
$$\pi^*\Omega^1_X\subset \pi^*\Omega^1_X(\log(\lfloor\Delta\rfloor))\subset \Omega^1(\pi,\Delta)\subset\pi^*\Omega^1_X(\log(\lceil\Delta\rceil)).$$
\end{rema}

\begin{rema}\label{rem:générateurs locaux du cotangent}
Dans des coordonnées adaptées au diviseur $\Delta$, nous pouvons donner une description explicite des générateurs locaux de $\Omega^1(\pi,\Delta)$. En effet, en utilisant les notations de la proposition \ref{prop:formes normales rev}, il est immédiat de constater que $\Omega^1(\pi,\Delta)$ est localement engendré par :
$$(w_1^{b_1-1}\dd w_1,\dots,w_k^{b_k-1}\dd w_k,\dd w_{k+1},\dots,\dd w_{n-j},w_{n-j+1}^{m_j-1}\dd w_{j+1},\dots, w_n^{m_1-1}\dd w_n).$$
En particulier, le faisceau $\Omega^1(\pi,\Delta)$ est localement libre et c'est pourquoi nous l'appelons \emph{fibré} cotangent orbifolde.
\end{rema}
Comme mentionné ci-dessus, le cotangent orbifolde est donc bien engendré par les images réciproques des formes multivaluées $\dfrac{\dd z_i}{z_i^{(1-b_i/a_i)}}$.

\begin{rema}\label{rem:determinant cotangent}
Une fois définie la notion de fibré cotangent orbifolde, il est naturel d'en considérer les puissances tensorielles $\Omega^1(\pi,\Delta)^{\otimes m}$, symétriques $\sym^m(\Omega^1(\pi,\Delta))$ ou alternées $\Omega^p(\pi,\Delta):=\bigwedge^p\Omega^1(\pi,\Delta)$. En particulier, le fibré $\Omega^1(\pi,\Delta)$ a pour déterminant :
$$\det(\Omega^1(\pi,\Delta))=\Omega^n(\pi,\Delta)=\pi^*(K_X+\Delta).$$
C'est une conséquence immédiate de la suite exacte (\ref{eq:defi cotangent}).
\end{rema}

Le résultat principal de \cite{CP15} est une extension au cas orbifolde du théorème \ref{th:quotient feuilletage K-psef}.
\begin{theo}\label{th:positivité orbifolde}
Soient $(X,\Delta)$ une orbifolde lisse avec $K_X+\Delta$ pseudo-effectif et $\pi:Y\to X$ un revêtement $\Delta$-adapté. Le fibré $\Omega^1(\pi,\Delta)$ a alors la propriété suivante :\\
pour tout entier $m\ge1$, toute classe mobile $\alpha\in\mob(X)$ et pour tout quotient
$$\Omega^1(\pi,\Delta)^{\otimes m}\twoheadrightarrow \mathcal{Q},$$
la pente de $\mathcal{Q}$ par rapport à $\pi^*\alpha$ est positive $\mu_{\pi^*\alpha}(\mathcal{Q})\ge0$.
\end{theo}
Les paragraphes suivants vont être consacrés à la démonstration du théorème ci-dessus, démonstration qui va nécessiter une étude détaillée de certains sous-faisceaux du fibré tangent orbifolde.

\subsection{Fibré tangent orbifolde et caractérisation de ses sections}

\begin{defi}\label{defi:tangent orbi}
Le fibré tangent orbifolde associé à $\pi$ est le dual du fibré cotangent orbifolde et est noté $T(\pi,\Delta)$. Il s'insère dans la suite d'inclusions :
$$\pi^*T_X(-\log(\lceil\Delta\rceil))\subset T(\pi,\Delta)\subset \pi^*T_X(-\log(\lfloor\Delta\rfloor))\subset \pi^*T_X.$$
Il est de plus engendré localement par les éléments :
$$(w_1^{a_1-b_1}\partial_1,\dots,w_k^{a_k-b_k}\partial_k,\partial_{k+1},\dots,\partial_{n-j},\partial_{n-j+1},\dots,\partial_n),$$
où nous avons noté $\partial_\ell:=\pi^*(\dfrac{\partial}{\partial z_\ell})$.
\end{defi}

\begin{rema}\label{rem:différentielle}
Les expressions des générateurs locaux rapportées dans la remarque \ref{rem:générateurs locaux du cotangent} et la définition \ref{defi:tangent orbi} peuvent sembler asymétriques. Cela vient du fait que, lorsque nous identifions (par exemple) $\pi^*\Omega^1_X$ à un sous-faisceau de $\Omega^1_Y$, nous faisons implicitement intervenir l'action de la différentielle $\dd\pi$. Ce n'est plus le cas lorsque nous nous plaçons sous l'angle du fibré tangent : l'écriture $\pi^*T_X$ désigne bel et bien l'image réciproque du fibré tangent par l'application $\pi$ (et dans cette situation duale, c'est le fibré $T_Y$ qui apparaît comme un sous-faisceau de $\pi^*T_X$ \emph{via} la différentielle de $\pi$).
\end{rema}

Il est intéressant de noter que, d'une certaine manière, les sections du fibré tangent orbifolde associé à $\pi$ admettent une caractérisation provenant de $X$. Pour cela, considérons des coordonnées adaptées comme celles provenant de la proposition \ref{prop:formes normales rev} et $v$ une section locale de $\pi^*T_X$. L'action du groupe d'isotropie local se faisant par multiplication par des racines de l'unité, nous en déduisons facilement que $v$ admet une décomposition unique de la forme :
\begin{equation}\label{eq:decomposition section locale}
v=\sum_{I\in \mathcal{I}}w^I\pi^*(v_I)
\end{equation}
où $v_I$ est un champ de vecteurs défini localement sur $X$ et $\mathcal{I}$ est l'ensemble des multi-indices $I=(i_1,\dots,i_n)$ vérifiant les conditions :
$$\left\{\begin{array}{cl}
0\le i_\ell<\max(b_\ell,1) & \mathrm{pour}\,\,\ell=1\dots k,\\
0\le i_{n-\ell+1}<m_\ell & \mathrm{pour}\,\,\ell=1\dots j,\\
i_\ell=0 & \mathrm{sinon.}
\end{array}\right.$$
De ce fait, pour chaque multi-indice $I$, on peut décomposer $v_I$ sous la forme
\begin{equation}\label{eq:decomposition chp vecteur}
v_I=\sum_{j=1}^n g^{(I)}_j(z)\frac{\partial}{\partial z_j}.
\end{equation}
Nous pouvons alors énoncer une caractérisation de $T(\pi,\Delta)$.
\begin{prop}\label{prop:caractérisation local cotangent orbi}
La section locale $v$ vérifie $v\in T(\pi,\Delta)$ si et seulement si $g^{(I)}_j$ est divisible par $z_j$ dès que $i_j<a_j-b_j$ (pour tout $j=1\dots k$).
\end{prop}
\begin{proof}
Il suffit de vérifier l'appartenance en codimension 1 et nous pouvons donc supposer que $\pi$ est de la forme $\pi(w_1,\dots,w_n)=(w_1^{a_1},w_2,\dots,w_n)$. Nous nous ramenons donc à un calcul en une variable. Soit donc une décomposition comme en (\ref{eq:decomposition section locale}) 
$$v=\sum_{i=0}^{a-1}w^i\pi^*(v_i)$$
avec l'analogue de (\ref{eq:decomposition chp vecteur}) : $v_i=z^{\alpha_i}f_i\dfrac{\partial}{\partial z}$ (et $\alpha_i$ est exactement l'ordre d'annulation en 0 de $v_i$). En combinant les deux expressions ci-dessus, on en déduit que
$$v=\sum_{i=0}^{a-1}w^{i+a\alpha_i}f_i(w^a)\pi^*\frac{\partial}{\partial z}.$$
Comme l'appartenance à $T(\pi,\Delta)$ se traduit par l'annulation de $v$ à l'ordre $a-b$, nous en déduisons que $v\in T(\pi,\Delta)$ si et seulement si $a\alpha_i+i\ge a-b$ pour tout $i=0\dots a-1$. Ceci se traduit bien par $\alpha_i=1$ si et seulement si $i<a-b$.
\end{proof}

\subsection{$\Delta$-Feuilletages}

Nous allons nous intéresser ici à certains sous-faisceaux du fibré tangent orbifolde que nous appellerons $\Delta$-feuilletages et établir une correspondance entre ceux-ci et les feuilletages sur $X$. Si $\F_X\subset T_X$ est un sous-faisceau du fibré tangent, nous pouvons considérer $\F:=T(\pi,\Delta)\cap \pi^*\F_X$ : c'est un sous-faisceau $G$-stable et saturé de $T(\pi,\Delta)$. Commençons par remarquer que la correspondance est bijective.
\begin{lemm}\label{lem:sous faisceau G-invariant}
Soit $\pi:Y\to X$ un revêtement fini et galoisien de groupe $G$ entre variétés lisses et soit $\mathcal{E}$ un faisceau sans torsion sur $X$. Si $\F\subset \pi^*\mathcal{E}$ est un sous-faisceau saturé et $G$-stable, il existe alors $\F_X$ un sous-faisceau saturé de $\mathcal{E}$ tel que $\F=\pi^*\F_X$.
\end{lemm}
\begin{proof}
En effet, pour tout faisceau sans torsion $\mathcal{H}$ sur $Y$ muni d'une action de $G$, nous pouvons considérer le faisceau\footnote{Signalons au passage que le foncteur $\pi_*^G$ est exact (voir par exemple \cite[lem. A.3]{GKKP}).} $\pi_*^G\mathcal{H}$ des sections $G$-invariantes de $\mathcal{H}$ et il est aisé de vérifier que le morphisme naturel
$$\pi^*(\pi_*^G\mathcal{H})\To \mathcal{H}$$
est une injection de faisceaux. Sur le lieu où $\pi$ est étale, ce morphisme est un isomorphisme et le noyau est donc un faisceau de torsion. Or, par platitude de $\pi$, le caractère sans torsion de $\mathcal{H}$ se propage à $\pi^*(\pi_*^G\mathcal{H})$ et le morphisme en question est bien injectif. Il suffit alors d'appliquer cette remarque à la situation $\F\subset \pi^*\mathcal{E}$ avec un quotient $\mathcal{Q}=\pi^*\mathcal{E}/\F$ sans-torsion. Nous avons alors (par platitude de $\pi$) :
$$\xymatrix{0\ar[r] & \F\ar[r] & \pi^*\mathcal{E}\ar[r] & \mathcal{Q}\ar[r] & 0\\
0\ar[r] & \pi^*(\pi_*^G\F)\ar[r]\ar@{^{(}->}[u] & \pi^*\mathcal{E}\ar[r]\ar@{=}[u] & \pi^*(\pi_*^G\mathcal{Q}).\ar@{^{(}->}[u] &
}$$
Par exactitude du diagramme ci-dessus, nous pouvons conclure que, dans cette situation, $\F=\pi^*(\F_X)$ avec $\F_X:=\pi_*^G(\F)$. Une fois encore, le fait que $\F_X$ soit saturé dans $\mathcal{E}$ est une conséquence de la platitude de $\pi$.
\end{proof}

Considérons alors $\F\subset T(\pi,\Delta)$ un sous-faisceau $G$-invariant et saturé dans $T(\pi,\Delta)$. Notons $\F^{s}$ son saturé dans $\pi^*T_X$. Comme ce dernier est encore $G$-invariant, nous pouvons lui appliquer le lemme \ref{lem:sous faisceau G-invariant} et en déduire que $\F^s=\pi^*(\F_X)$ pour un certain sous-faisceau $\F_X$ (saturé) de $T_X$. Sur la variété $X$, nous disposons du crochet de Lie qui induit un morphisme $\mathcal{O}_X$-linéaire :
$$\LL_{\F_X}:\bigwedge^2\F_X\stackrel{[\cdot,\cdot]}{\To}T_X/\F_X.$$
Nous pouvons considérer l'image réciproque de cette application par $\pi$ et nous en déduisons un morphisme $\mathcal{O}_Y$-linéaire :
\begin{equation}\label{eq:image réciproque crochet de Lie}
\LL_{\F^s}:=\pi^*\LL_{\F_X}:\bigwedge^2\F^s\To \pi^*T_X/\F^s.
\end{equation}
Le lemme suivant (dont nous reportons la démonstration à la fin du présent paragraphe) est crucial et montre que la restriction de cette application à $\F$ vérifie une propriété similaire à celle du crochet de Lie sur $X$.
\begin{lemm}\label{lem:factorisation crochet de Lie}
La restriction de $\LL_{\F^s}$ à $\F$ a son image dans $T(\pi,\Delta)/\F$ :
$$\xymatrix{\bigwedge^2\F\ar@{^{(}->}[r]\ar[rrd] & \bigwedge^2\F^s\ar[r]^{\LL_{\F^s}} & \pi^*T_X/\F^s\\
&&T(\pi,\Delta)/\F.\ar@{^{(}->}[u]
}$$
\end{lemm}
\begin{rema}\label{rem:inclusion quotients}
Comme nous avons supposé $\F$ saturé dans $T(\pi,\Delta)$, nous avons bien l'identification
$$\F=T(\pi,\Delta)\cap \F^s$$
ainsi que l'inclusion entre les quotients $T(\pi,\Delta)/\F$ et $\pi^*T_X/\F^s$.
\end{rema}
Le lemme \ref{lem:factorisation crochet de Lie} motive la définition suivante.
\begin{defi}\label{defi:Delta-feuilletage}
Un sous-faisceau $\F\subset T(\pi,\Delta)$ $G$-stable et saturé dans $T(\pi,\Delta)$ sera appelé un $\Delta$-feuilletage si l'application induite par le crochet de Lie
$$\bigwedge^2\F\To T(\pi,\Delta)/\F$$
est nulle.
\end{defi}
La correspondance annoncée au début de ce paragraphe est alors résumée dans la proposition suivante.
\begin{prop}\label{prop:descente feuilletage}
Soient $\F\subset T(\pi,\Delta)$ un $\Delta$-feuilletage et $\F^{s}$ son saturé dans $\pi^*T_X$. Il existe alors un sous-faisceau $\F_X\subset T_X$ vérifiant :
\begin{enumerate}
\item $\F^{s}=\pi^*\F_X$ avec $\F_X$ saturé dans $T_X$.
\item $\F_X$ est stable par crochet de Lie : $\F_X$ définit un feuilletage sur $X$.
\item Les fibrés canoniques de $\F$ et $\F_X$ vérifient :
$K_\F:=\det(\F^*)=\pi^*(K_{\F_X}+\Delta^{hor})$.
\end{enumerate}
\end{prop}
\begin{proof}
Le premier point n'étant autre que la conclusion du lemme \ref{lem:sous faisceau G-invariant}, concentrons nous sur les points restants. Le lemme \ref{lem:factorisation crochet de Lie} montre que l'application induite par le crochet de Lie
$$\LL_{\F^s}:\bigwedge^2\F^{s}\To\pi^*T_X/\F^{s}$$
est génériquement nulle puisque nulle sur $\bigwedge^2\F$. Le quotient $\pi^*T_X/\F^{s}$ étant sans torsion, l'application ci-dessus est donc identiquement nulle et cela signifie exactement que $\F_X$ est stable par crochet de Lie.

Il nous reste à relier les fibrés canoniques de $\F_X$ et de $\F$. Pour cela, nous utilisons le point de vue des formes comme rappelé dans le paragraphe \ref{sub:rappel feuilletage} : le feuilletage $\F_X$ est donné par une section
$$\omega\in H^0(X,\Omega^q_X\otimes \det(N_{\F_X}))$$
dont le lieu des zéros est de codimension au moins deux dans $X$. Le fibré canonique de $\F=\pi^*(\F_X)\cap T(\pi,\Delta)$ sera alors linéairement équivalent à
\begin{equation}\label{eq:canonique feuilletage orbi}
\det(\Omega^1(\pi,\Delta))+\pi^*\det(N_{\F_X})-Z(\omega_\F)
\end{equation}
où $Z(\omega_\F)$ est la partie divisorielle des zéros de $\omega_\F:=\pi^*(\omega)$ \emph{vue comme une section} de $\Omega^q(\pi,\Delta)\otimes \pi^*\det(N_{\F_X})$. Une fois cette observation faite, le calcul est aisé : 
\begin{enumerate}
\item au-dessus d'une composante transverse à $\F_X$, la forme $\omega_\F$ ne s'annule pas en codimension 1.
\item au-dessus d'un point général du lieu de branchement de $\pi$ qui n'est pas sur $\Delta$, la forme ne s'annule pas non plus (car $\Omega^q(\pi,\Delta)$ prend en compte la ramification additionnelle).
\item si $\Delta_1=\set{z_1=0}$ est une composante invariante par $\F_X$, la forme $\omega$ s'écrit alors
$$\omega=\dd z_1\wedge \dd z_2\wedge\dots\wedge \dd z_q$$
(et $\pi$ ne ramifie pas au-dessus de $z_j$, $j=2\dots q$). Nous avons alors :
$$\omega_\F=w_1^{a_1-1}\dd w_1\wedge \dd w_2\wedge\dots\wedge \dd w_q=w_1^{a_1-b_1}\left(w_1^{b_1-1}\dd w_1\wedge \dd w_2\wedge\dots\wedge \dd w_q\right)$$
où l'élément entre parenthèses est une section de $\Omega^q(\pi,\Delta)$.
\end{enumerate}
Au vu de la formule (\ref{eq:canonique feuilletage orbi}), le canonique de $\F$ est donné par :
\begin{align*}
K_\F&=\sum_{\Delta_i\subset \Delta^{inv}}\pi^*((\frac{b_i}{a_i}-1)\Delta_i)+\det(\Omega^1(\pi,\Delta))+\pi^*\det(N_{\F_X})\\
&=\pi^*(-\Delta^{inv}+K_X+\Delta+\det(N_{\F_X}))=\pi^*(K_{\F_X}+\Delta^{hor}).
\end{align*}
\end{proof}
\noindent Nous remercions Stéphane Druel de nous avoir communiqué cette dernière observation.

Avant de clore ce paragraphe, nous devons encore démontrer le lemme \ref{lem:factorisation crochet de Lie}. La subtilité vient de ce que l'on peut définir \emph{localement} un relèvement du crochet de Lie à $\pi^*T_X$ tout entier mais il faut bien prendre garde au fait que cette opération ne se recolle absolument pas puisqu'elle dépend d'un choix de base (locale) de $\mathcal{O}_Y$ sur $\mathcal{O}_X$. Nous nous plaçons donc dans des coordonnées adaptées à $\pi$ (données par la proposition \ref{prop:formes normales rev}) et définies sur un ouvert\footnote{Pour la topologie usuelle.} $U\subset Y$.
\begin{defi}\label{defi:orbi crochet de lie}
Soient $u$ et $v$ deux sections locales de $(\pi^*T_X)_{\mid U}$ et écrivons leur décomposition comme en (\ref{eq:decomposition section locale}) ci-dessus :
$$u=\sum_{I\in \mathcal{I}}w^I\pi^*(u_I)\quad\textrm{et}\quad v=\sum_{I\in \mathcal{I}}w^I\pi^*(v_I).$$
Nous posons alors :
$$\mathcal{L}^{loc}_\pi(u,v)=\sum_{I,J\in \mathcal{I}}w^{I+J}\pi^*\left([u_I,v_J]\right).$$
\end{defi}
\noindent Malgré la lourdeur des notations, nous insistons sur la restriction des faisceaux considérés à l'ouvert $U$ pour bien garder en mémoire que cette opération a un caractère éminemment local.
\begin{lemm}\label{lem:crochet local induit ce qu'il faut}
L'application $\LL^{loc}_\pi$ est un relèvement \emph{local} $\CC$-linéaire
$$\LL_\pi^{loc}:(\pi^*T_X)_{\mid U}\times (\pi^*T_X)_{\mid U}\To (\pi^*T_X)_{\mid U}$$
du morphisme $\mathcal{O}_Y$-linéaire $\LL_{\F^s}$ défini en (\ref{eq:image réciproque crochet de Lie}) pour tout faisceau $\F^s=\pi^*(\F_X)$.
\end{lemm}
\begin{proof}
Là encore, en raisonnant coordonnées par coordonnées, on peut se limiter au cas de $\pi(w_1,\dots,w_n)=(w_1^{a},w_2,\dots,w_n)$ et il suffit de vérifier que
$$\LL^{loc}_\pi(w_1^k\pi^*(u_X),\pi^*(v_X))=w_1^k\LL^{loc}_\pi(\pi^*(u_X),\pi^*(v_X))\quad\mathrm{mod}\,\pi^*\F_X$$
pour tout $k\ge0$. Pour cela, effectuons la division euclidienne de $k$ par $a$ et écrivons $k=aq+r$. Il vient alors :
\begin{align*}
\LL^{loc}_\pi(w_1^k\pi^*(u_X),\pi^*(v_X))&=\LL^{loc}_\pi(w_1^{aq+r}\pi^*(u_X),\pi^*(v_X))=\LL^{loc}_\pi(w_1^r\pi^*(z_1^qu_X),\pi^*(v_X))\\
&=w_1^r\pi^*[z_1^qu_X,v_X]=w_1^r\pi^*\left(z_1^q[u_X,v_X]+\varphi u_X \right)\\
&=w_1^{r+aq}\LL^{loc}_\pi(\pi^*(u_X),\pi^*(v_X))+w_1^r\pi^*(\varphi u_X)
\end{align*}
avec $\varphi=v_X(z_1^q)$ une fonction sur $X$. Ceci montre bien le résultat.
\end{proof}

Il nous reste à montrer que le fibré tangent orbifolde est fermé pour cette opération.
\begin{prop}\label{prop:tangent orbi=feuilletage}
Le fibré $T(\pi,\Delta)_{\mid U}\subset(\pi^*T_X)_{\mid U}$ est stable par $\LL^{loc}_\pi$ :
$$\LL^{loc}_\pi\left(T(\pi,\Delta)_{\mid U}\times T(\pi,\Delta)_{\mid U}\right)\subset T(\pi,\Delta)_{\mid U}.$$
\end{prop}
\begin{proof}
Nous allons exploiter la caractérisation obtenue à la proposition \ref{prop:caractérisation local cotangent orbi}. Soit donc $u$ et $v$ deux sections locales de $T(\pi,\Delta)$ que nous écrivons
$$u=\sum_{I\in \mathcal{I}}w^I\pi^*(u_I)\quad\textrm{et}\quad v=\sum_{J\in \mathcal{I}}w^J\pi^*(v_J).$$
D'après la proposition \ref{prop:caractérisation local cotangent orbi}, les champs de vecteurs $u_I$ et $v_J$ s'écrivent
$$u_I=\sum_{j=1}^n z_j^{\alpha_{I,j}}f^{(I)}_j\frac{\partial}{\partial z_j}\quad\textrm{et}\quad v_J=\sum_{j=1}^n z_j^{\beta_{J,j}}g^{(J)}_j\frac{\partial}{\partial z_j}$$
avec $\alpha_{I,j}=1$ dès que $I(j)<a_j-b_j$ (resp. $\beta_{J,j}=1$ dès que $J(j)<a_j-b_j$). Or, l'expression du crochet de Lie de $u$ et $v$ s'écrit :
\begin{align*}
\LL_\pi^{loc}(u, v)&=\sum_{I,J\in \mathcal{I}}w^{I+J}\pi^*([u_I,v_J])\\
&=\sum_{I,J\in \mathcal{I}}w^{I+J}\pi^*\left(\sum_{j,k=1}^n[z_j^{\alpha_{I,j}}f^{(I)}_j\frac{\partial}{\partial z_j},z_k^{\beta_{J,k}}g^{(J)}_k\frac{\partial}{\partial z_k}]\right)
\end{align*}
et il nous faut donc vérifier que le crochet de Lie vérifie les bonnes conditions d'annulation. Mais c'est évident puisque
$$[z_j^{\alpha_{I,j}}f^{(I)}_j\frac{\partial}{\partial z_j},z_k^{\beta_{J,k}}g^{(J)}_k\frac{\partial}{\partial z_k}]=\left\{\begin{array}{cc}
z_j^{\alpha_{I,j}}z_k^{\beta_{J,k}}\frac{\partial f^{(I)}_j}{\partial z_k}g^{(J)}_k\frac{\partial}{\partial z_j} - z_j^{\alpha_{I,j}}z_k^{\beta_{J,k}}\frac{\partial g^{(J)}_k}{\partial z_j}f^{(I)}_j\frac{\partial}{\partial z_k} & \textrm{si }j\neq k\\
&\\
z_j^{\alpha_{I,j}+\beta_{J,j}}(\frac{\partial f^{(I)}_j}{\partial z_j}g^{(J)}_j-\frac{\partial g^{(J)}_j}{\partial z_j}f^{(I)}_j)\frac{\partial}{\partial z_j}&\textrm{sinon.}
\end{array}\right.
$$
\end{proof}

\begin{proof}[Démonstration du lemme \ref{lem:factorisation crochet de Lie}]
Il s'agit de la combinaison du lemme \ref{lem:crochet local induit ce qu'il faut} et de la proposition \ref{prop:tangent orbi=feuilletage}.
\end{proof}

\subsection{Démonstration du théorème principal et conséquences}

\noindent Nous avons maintenant effectué tous les préparatifs nécessaires à la
\begin{proof}[Démonstration du théorème \ref{th:positivité orbifolde}]
Raisonnons par l'absurde et supposons que $\Omega^1(\pi,\Delta)^{\otimes m}$ admette un quotient dont la pente est strictement négative par rapport à une classe mobile $\pi^*\alpha$. Par dualité, cela signifie $\mu_{\pi^*\alpha}^{max}(T(\pi,\Delta)^{\otimes m})>0$ et le théorème \ref{th:mumax produit tensoriel} montre qu'il en est de même pour $T(\pi,\Delta)$. Considérons alors le déstabilisant maximal $\F\subset T(\pi,\Delta)$ pour cette classe $\pi^*\alpha$. Par unicité, le faisceau $\F$ est donc saturé dans $T(\pi,\Delta)$ et $G$-stable. De plus, toujours pour des raisons de pentes par rapport à $\pi^*\alpha$, tout morphisme
$$\bigwedge^2\F\To T(\pi,\Delta)/\F$$
doit être nul. Nous pouvons donc appliquer la proposition \ref{prop:descente feuilletage} : le saturé $\F^s$ de $\F$ dans $\pi^*T_X$ s'écrit donc comme l'image réciproque $\F^s=\pi^*(\F_X)$ avec $\F_X$ un feuilletage sur $X$. D'autre part, en utilisant la remarque \ref{rem:pente saturation}, nous obtenons la suite d'inégalité des pentes
$$\mu^{min}_\alpha(\F_X)\ge\mu^{min}_{\pi^*\alpha}(\F^s)\ge\mu^{min}_{\pi^*\alpha}(\F)>0,$$
ce qui montre que le théorème \ref{th:critère mobile} s'applique : $\F_X$ est algébriquement intégrable. Comme nous avons supposé que $K_X+\Delta$ est pseudo-effectif, le théorème \ref{th:positivité feuilletages} nous permet d'affirmer que $K_{\F_X}+\Delta^{hor}$ l'est également. Nous aboutissons ainsi à une contradiction puisque, à nouveau d'après la proposition \ref{prop:descente feuilletage},
$$0>\pi^*\alpha\cdot K_\F=\pi^*\alpha\cdot\pi^*(K_{\F_X}+\Delta^{hor})=\deg(\pi)\alpha\cdot (K_{\F_X}+\Delta^{hor})\ge 0.$$
\end{proof}

Comme annoncé plus haut dans ce texte, le théorème \ref{th:positivité orbifolde} permet de montrer que la conclusion du théorème \ref{th:positivité feuilletages} reste vraie sans l'hypothèse d'intégrabilité algébrique.
\begin{coro}\label{cor:feuilletage tordu}
Si $(X,\Delta)$ est une orbifolde avec $K_X+\Delta$ pseudo-effectif, alors le fibré canonique tordu $K_\F+\Delta^{hor}$ est pseudo-effectif pour tout feuilletage $\F$ sur $X$.
\end{coro}
\begin{proof}
Le $\Delta$-feuilletage $\F_\Delta:=T(\pi,\Delta)\cap \pi^*\F$ fournit (par passage au dual) un quotient de $\Omega^1(\pi,\Delta)$ et nous pouvons appliquer le théorème \ref{th:positivité orbifolde} : $K_{\F_\Delta}\cdot \pi^*\alpha\ge 0$ pour toute classe mobile $\alpha$. D'après la proposition \ref{prop:descente feuilletage}, le fibré canonique de $\F_\Delta$ est donné par $K_{\F_\Delta}=\pi^*(K_\F+\Delta^{hor})$ et l'inégalité ci-dessus devient : $\deg(\pi)(K_\F+\Delta^{hor})\cdot\alpha\ge0$ et le diviseur $K_\F+\Delta^{hor}$ est donc bien pseudo-effectif.
\end{proof}

Le théorème \ref{th:positivité orbifolde} permet également d'établir un critère assurant qu'une paire est de type général, critère qui sera exploité dans la troisième et dernière partie.
\begin{theo}\label{th:critère type général}
Soit $(X,\Delta)$ une orbifolde et supposons qu'il existe $\F$ un faisceau sur $X$ avec $\det(\F)$ \emph{big}\footnote{Un fibré en droites $L$ sur $X$ est dit \emph{big} si $\kappa(L)=\dim(X)$. La terminologie, qu'elle soit francophone ou anglophone, n'est pas pleinement satisfaisante.} et un revêtement $\Delta$-adapté muni d'une injection :
$$\pi^*\F\hookrightarrow \Omega^1(\pi,\Delta)^{\otimes m}$$
pour un certain entier $m\ge1$. La paire $(X,\Delta)$ est alors de log-type général.
\end{theo}
\begin{proof}
Si nous savons montrer que $K_X+\Delta$ est pseudo-effectif, la démonstration est alors immédiate. En effet, dans cette situation, le quotient
$$\mathcal{Q}:=\Omega^1(\pi,\Delta)^{\otimes m}/\pi^*\F$$
vérifie alors $\mu_{\pi^*\alpha}(\mathcal{Q})\ge 0$. Mais nous avons alors :
$$N\pi^*(K_X+\Delta)=\det(\Omega^1(\pi,\Delta)^{\otimes m})=\det(\pi^*\F)+\det(\mathcal{Q})=\pi^*\det(\F)+\det(\mathcal{Q})$$
avec $N:=mn^{m-1}$. Pour tout $\alpha$ classe mobile sur $X$, nous en déduisons :
\begin{align*}
\deg(\pi)\left(N(K_X+\Delta)-\det(\F)\right)\cdot\alpha=&\pi^*\left(N(K_X+\Delta)-\det(\F)\right)\cdot\pi^*\alpha\\
&=\rk(\mathcal{Q})\mu_{\pi^*\alpha}(\mathcal{Q})\ge0.
\end{align*}
Cela montre que la différence $N(K_X+\Delta)-\det(\F)$ est pseudo-effectif et donc que $K_X+\Delta$ est \emph{big}.

Considérons alors $H$ un diviseur ample sur $X$ et le seuil suivant :
$$\tau:=\inf\set{t>0\mid K_X+\Delta+tH\in\psef(X)}.$$
Si $(t_k)_{k\ge1}$ est une suite de rationnels qui converge vers $\tau$, nous pouvons considérer pour chaque valeur de $k$ un revêtement $\Delta_k:=\Delta+t_kH$-adapté qui de plus domine $\pi$. Nous sommes donc dans la situation
$$\xymatrix{Y_k\ar[rr]^{\pi_k}\ar[rd] && X\\
&Y\ar[ru]_\pi
}$$
où $\pi_k$ factorise à travers $\pi$. En particulier, nous avons $\pi_k^*\F\subset \Omega^1(\pi_k,\Delta_k)^{\otimes m}$. Le même calcul que celui effectué ci-dessus montre que
$$N(K_X+\Delta_k)-\det(\F)$$
est pseudo-effectif et, en passant à la limite, le diviseur
$$N(K_X+\Delta+\tau H)-\det(\F)$$
est également pseudo-effectif (le cône $\psef(X)$ est fermé). Le diviseur $K_X+\Delta+\tau H$ est donc \emph{big} et, par définition de $\tau$, cela implique que le seuil $\tau$ doit être nul : le diviseur $K_X+\Delta$ est donc bien \emph{big} également.
\end{proof}
L'observation du fait que le théorème \ref{th:positivité orbifolde} permet d'éviter le recours aux résultats de finitude de \cite{BCHM} est due à Behrouz Taji.

\begin{rema}
Il est important de souligner ici l'importance du cadre orbifolde, importance qui apparaît pleinement dans la démonstration ci-dessus. En effet, le théorème \ref{th:critère type général} peut s'énoncer dans les cas classiques $\Delta=0$ ou $\Delta$ réduit et l'aspect orbifolde est alors absent, aussi bien des hypothèses que de la conclusion. Cependant, la démonstration montre qu'il est primordial d'élargir le cadre classique et d'autoriser des coefficients rationnels pour les composantes de $\Delta$.
\end{rema}

Notons que ce résultat peut se reformuler en termes de fibrés sur $X$. En effet, considérons les sections invariantes :
$$\mathsf{S}^m\Omega^p_X(\Delta):=\pi^G_*\left(\sym^m\Omega^p(\pi,\Delta)\right)\quad\mathrm{et}\quad \mathsf{T}^m\Omega^p_X(\Delta):=\pi^G_*\left(\Omega^p(\pi,\Delta)^{\otimes m}\right).$$
Ces faisceaux ont été considérés dans \cite[\S 2.5]{Ca11j} : il n'est pas difficile de constater qu'ils sont localement libres et d'en exhiber des générateurs locaux (\emph{cf. loc. cit.}). En particulier, ces faisceaux ne dépendent pas du revêtement $\Delta$-adapté $\pi$ utilisé pour les définir. Le théorème \ref{th:critère type général} peut se réécrire de la façon suivante.
\begin{coro}\label{cor:critère type général}
Soient $(X,\Delta)$ une orbifolde et $\F$ un faisceau avec $\det(\F)$ \emph{big} sur $X$. S'il existe un entier $m\ge 1$ et une injection de faisceaux :
$$\F\hookrightarrow \mathsf{T}^m\Omega^1_X(\Delta),$$
l'orbifolde $(X,\Delta)$ est alors de type général.
\end{coro}
\begin{rema}\label{rem:faisceaux diff orbi sur X}
Il faut cependant garder en mémoire que, même si $K_X+\Delta$ est pseudo-effectif, les faisceaux $\mathsf{T}^m\Omega^1_X(\Delta)$ ne vérifient pas la conclusion du théorème \ref{th:positivité orbifolde}. En effet, la structure orbifolde sur $\PP^1$ composée de 4 points affectés de multiplicités 2 a un fibré canonique trivial. En revanche, il est immédiat de constater que $\mathsf{T}^1\Omega^1_{\PP^1}(\Delta)=\Omega^1_{\PP^1}$ et n'est donc pas pseudo-effectif.
\end{rema}

\section{Fibration du c\oe ur et application aux familles de variétés canoniquement polarisées}
Il est bien entendu impossible de résumer les articles \cite{Ca04} et \cite{Ca11j} en quelques pages. Nous souhaitons donner ici un aperçu des principes de classification des variétés projectives envisagées dans les travaux de F. Campana. Ceux-ci sont en grande partie condensés dans l'existence d'une fibration canoniquement associée à toute variété (ou orbifolde) et qui la scinde en deux géométries antithétiques : les fibres sont \emph{spéciales} alors que la base est de type général. L'utilisation de cette fibration permettra également une formulation unifiée des conjectures de Viehweg et Campana concernant les familles de variétés canoniquement polarisées.

\subsection{Orbifoldes spéciales et fibration du c\oe ur}

\subsubsection{Faisceaux différentiels et dimension canonique orbifolde}
Nous définissons dans un premier temps une dimension de Kodaira prenant en compte la présence d'une structure orbifolde.
\begin{defi}\label{defi:dimension de Kodaira orbi}
Soient $(X,\Delta)$ une orbifolde et $L\subset \Omega^p_X$ un faisceau saturé de rang 1 (un tel faisceau sera appelé faisceau différentiel). La dimension canonique (ou de Kodaira) de $L$ relativement à la structure orbifolde $\Delta$ est par définition :
$$\kappa_\Delta(X,L):=\kappa(Y,(\pi^*L)^{sat})$$
où $\pi:Y\to X$ est un revêtement $\Delta$-adapté et où la saturation se réfère au faisceau $\Omega^p(\pi,\Delta)$.
\end{defi}
\begin{exem}\label{ex:dimension canonique orbifolde}
Pour $L:=K_X$, nous obtenons $\kappa_\Delta(X,K_X)=\kappa(X,K_X+\Delta)$. Nous noterons également $\kappa(X,\Delta)$ cette quantité.
\end{exem}
\begin{rema}\label{rem:croissance dimension canonique}
Si $\Delta^1\le \Delta^2$ sont deux structures orbifoldes dont la différence est effective, les dimensions canoniques d'un faisceau différentiel $L$ vérifient naturellement : $\kappa_{\Delta^1}(X,L)\le\kappa_{\Delta^2}(X,L)$. Il est en effet possible de choisir un revêtement adapté pour $\Delta^2$ qui en domine un adapté pour $\Delta^1$.
\end{rema}

Cette notion de dimension canonique est particulièrement adaptée à l'étude des fibrations définies sur $X$. En effet, si $f:X\dashrightarrow Z$ est une fibration rationnelle avec $Z$ lisse, il existe un  sous-faisceau de rang 1 de $\Omega^p_X$ avec $p=\dim(Z)$ et qui coïncide avec $f^*K_Z$ au point général de $X$ : si $\varphi:\hat{X}\to X$ est une modification rendant $\hat{f}:\hat{X}\to Z$ régulière, l'expression $\mathsf{L}_f=\varphi_*(\hat{f}^*K_Z))^{sat}\subset \Omega^p_X$ définit bien un sous-faisceau de rang 1 de $\Omega^p_X$ qui ne dépend pas des modèles birationnels choisis (que ce soit pour $\hat{X}$ ou $Z$).
\begin{defi}\label{defi:kappa fibration}
Dans la situation ci-dessus, nous posons $\kappa_\Delta(f):=\kappa_\Delta(X,\mathsf{L}_f)$. Si $\kappa_\Delta(f)=\dim(Z)$, nous dirons que $f$ est de type général.
\end{defi}
Il est remarquable de noter que le fibré $\mathsf{L}_f$ est en quelque sorte l'image réciproque du fibré canonique d'une certaine structure orbifolde sur $Z$ (et ces considérations interviennent même dans le cas $\Delta=0$). Pour que cette construction ait un sens, il convient d'abord de se placer sur un bon modèle de la fibration $f$. Nous renvoyons à \cite[\S 4.1]{Ca11j} pour les détails concernant l'énoncé suivant.
\begin{prop}\label{prop:existence base orbifolde}
Si $f:X\dashrightarrow Z$ est une fibration rationnelle et $\Delta$ une structure orbifolde sur $X$, il existe alors un modèle birationnel de $Z$ (toujours noté $Z$) et une structure orbifolde $\Delta_Z:=\Delta_Z(f,\Delta)$ sur $Z$ qui de plus vérifie
$$\kappa(Z,\Delta_Z)=\kappa_\Delta(f).$$
La paire $(Z,\Delta_Z)$ sera appelée la \emph{base orbifolde} de $f$ induite par $\Delta$.
\end{prop}
Finissons cette discussion avec le résultat suivant qui établit une borne sur la dimension canonique des faisceaux différentiels ainsi qu'une correspondance entre certains de ces derniers et fibrations de type général.
\begin{theo}\label{th:inégalité bogomolov}
Soient $(X,\Delta)$ une orbifolde et $L\subset \Omega^p_X$ un faisceau différentiel. La dimension de Kodaira orbifolde de $L$ vérifie
$$\kappa_\Delta(X,L)\le p$$
et, dans le cas d'égalité, il existe une fibration de type général $f:X\dashrightarrow Z$ telle que $L=\mathsf{L}_f$.
\end{theo}
Ce résultat \cite[cor. 3.13]{Ca11j} est une généralisation au cas orbifolde d'un théorème dû à Bogomolov \cite{Bog} dans le cas $\Delta=0$.

\subsubsection{Orbifoldes spéciales : définitions et exemples}
Nous pouvons maintenant donner la définition d'une orbifolde spéciale.
\begin{defi}\label{def:orbifolde spéciale}
Une orbifolde $(X,\Delta)$ est dite \emph{spéciale} si pour toute fibration rationnelle $f:X\dashrightarrow Z$ (avec $\dim(Z)>0$), l'inégalité suivante est vérifiée :
$$\kappa_\Delta(f)<\dim(Z).$$
En d'autres termes, $(X,\Delta)$ est spéciale si elle n'admet pas de fibrations de type général.
\end{defi}
\begin{rema}\label{rem:spécial=pas de Bog}
Le théorème \ref{th:inégalité bogomolov} montre qu'une orbifolde $(X,\Delta)$ est spéciale si et seulement si $\kappa_\Delta(X,L)< p$ pour tout faisceau différentiel $L$ sur $X$.
\end{rema}
\begin{exem}\label{ex:cas des courbes}
Il est immédiat de vérifier qu'une courbe $(C,\Delta)$ de genre $g$ est spéciale si et seulement si $2g-2+\deg(\Delta)\le0$. En particulier, dans le cas où $\Delta$ est réduit, les courbes orbifoldes spéciales sont exactement $\PP^1$, $\CC$, $\CC^*$ et les courbes elliptiques.
\end{exem}

Le théorème suivant est une conséquence du théorème de faible positivité \ref{th:campana-fujino} et généralise les résultats classique de \cite{V83,K81}. Il va nous permettre en particulier d'exhiber de nouveaux exemples d'orbifoldes spéciales (mais il est aussi un des outils essentiels de la construction du c{\oe}ur).
\begin{theo}\label{th:Cnm orb}
Si $f:(X,\Delta)\dashrightarrow Z$ est une fibration de type général, les dimensions canoniques vérifient l'inégalité :
$$\kappa(X,\Delta)\ge \kappa(X_z,\Delta_z)+\dim(Z)$$
où $\kappa(X_z,\Delta_z)$ désigne la restriction de la structure orbifolde $\Delta$ à la fibre générale de $f$.
\end{theo}
\begin{rema}\label{rem:type général implique presque holomorphe}
Il est établi dans \cite[th. 9.17]{Ca11j} qu'une fibration de type général est automatiquement presque holomorphe (son lieu d'indétermination ne rencontre pas la fibre générale) et la fibre générale $(X_z,\Delta_z)$ est donc bien définie.
\end{rema}
Voici les nouvelles classes d'exemples annoncées ci-dessus.
\begin{exem}\label{ex:orbifolde spéciale}
Une orbifolde $(X,\Delta)$ vérifiant $\kappa(X,\Delta)=0$ est spéciale. Une orbifolde $(X,\Delta)$ de Fano (c'est-à-dire avec $-(K_X+\Delta)$ ample) est également spéciale.
\end{exem}
\begin{proof}
Si $f:(X,\Delta)\dashrightarrow Z$ est une fibration de type général dont l'espace total vérifie $\kappa(X,\Delta)=0$, les fibres générales ont également une dimension canonique positive : $\kappa(X_z,\Delta_z)\ge0$. Mais ceci est incompatible avec la sur-additivité énoncée dans le théorème \ref{th:Cnm orb} si $\dim(Z)>0$ et $(X,\Delta)$ est donc spéciale.

Si $-(K_X+\Delta)$ est ample, nous considérons $\Delta':=\Delta+\dfrac{1}{N}H$ où $H$ est une section hyperplane générale de $\abs{-N(K_X+\Delta)}$. Pour cette nouvelle structure orbifolde, nous avons $K_X+\Delta'\sim_\QQ 0$ et l'orbifolde $(X,\Delta')$ est donc spéciale d'après l'exemple précédent. Les remarques \ref{rem:croissance dimension canonique} et \ref{rem:spécial=pas de Bog} montrent que $(X,\Delta)$ est également spéciale.
\end{proof}
Dans le cas $\Delta=0$, l'annulation des tenseurs holomorphes sur une variété rationnellement connexe montre qu'une telle variété doit être spéciale (et ceci généralise donc le cas Fano). La notion même de \emph{connexité rationnelle orbifolde} n'est pas si claire (\emph{cf.} la discussion dans \cite[\S 6]{Ca11j}) et une variante numérique est proposée dans \emph{loc. cit.}
\begin{defi}\label{def:kappa+}
Si $(X,\Delta)$ est une orbifolde lisse, posons\footnote{Une orbifolde $(X,\Delta)$ vérifiant $\kappa_+(X,\Delta)=-\infty$ est appelée $\kappa$-\emph{rationnellement connexe} dans \cite[\S 7.5]{Ca11j}.} :
$$\kappa_+(X,\Delta):=\max\left(\kappa_\Delta(f)\mid f:X\dashrightarrow Z\right).$$
\end{defi}
Il est évident qu'une orbifolde $(X,\Delta)$ vérifiant $\kappa_+(X,\Delta)=-\infty$ est spéciale. Il semblerait d'ailleurs que cette classe, combinée avec les orbifoldes vérifiant $\kappa=0$, permette de reconstruire la classe des orbifoldes spéciales. En effet, F. Campana a montré \cite[cor. 11.4]{Ca11j} que, en admettant une version généralisée de la conjecture $C_{n,m}$ d'Iitaka, toute orbifolde spéciale est obtenue comme l'espace total d'une tour de fibrations dont les fibres orbifoldes vérifient $\kappa=0$ ou $\kappa_+=-\infty$.

\subsubsection{Réduction du c\oe ur}
Comme nous l'avons mentionné plus haut, l'introduction de la classe des orbifoldes spéciales permet en particulier d'avoir une image assez simple de la géométrie des orbifoldes : c'est une combinaison d'une partie spéciale et d'une autre de type général. L'existence de la réduction du c{\oe}ur a été établie dans \cite[th. 3.3]{Ca04} dans le cas des variétés et dans \cite[th. 10.1]{Ca11j} pour les orbifoldes.
\begin{theo}\label{th:decomposition du coeur}
Soit $(X,\Delta)$ une orbifolde. Il existe alors une unique fibration rationnelle
$$c_{(X,\Delta)}:(X,\Delta)\dashrightarrow C(X,\Delta)$$
vérifiant :
\begin{itemize}
\item si $(X,\Delta)$ est spéciale, $C(X,\Delta)$ est un point.
\item sinon, $c_{(X,\Delta)}$ est de type général et ses fibres orbifoldes générales sont spéciales.
\end{itemize}
Cette fibration est de plus maximale par rapport aux fibrations de type général : si $f:(X,\Delta)\dashrightarrow Z$ est une fibration de type général, $f$ se factorise à travers $c_{(X,\Delta)}$ :
$$\xymatrix{(X,\Delta)\ar@{-->}[rr]^f\ar@{-->}[rd]_{c_{(X,\Delta)}} & & Z\\
&C(X,\Delta).\ar@{-->}[ru] & 
}$$
La fibration $c_{(X,\Delta)}$ est appelée le \emph{c\oe ur} de l'orbifolde $(X,\Delta)$.
\end{theo}
Mentionnons à titre de remarque que le théorème \ref{th:Cnm orb} d'additivité orbifolde pour les fibrations de type général est l'outil essentiel pour établir le fait que le c{\oe}ur est une fibration de type général si $X$ n'est pas spécial.

\subsection{Hyperbolicité algébrique de l'espace des modules des variétés canoniquement polarisées}
\`A nouveau, nous ne donnons qu'un bref résumé des résultats de Viehweg \cite{livreVie} et Viehweg-Zuo \cite{VZ00} concernant l'espace des modules des variétés canoniquement polarisées et les formes différentielles provenant de ce dernier.
\subsubsection{Généralités sur l'espace des modules}
Dans son imposant travail \cite{livreVie}, Viehweg construit un espace de modules pour les variétés canoniquement polarisées. Considérons pour cela le foncteur de modules suivant :
$$\mathcal{M}_h:\left\{\begin{array}{ccl}
\mathcal{S}ch & \To & \mathcal{E}ns\\
S & \mapsto & \set{f:X\to S\mid f\,\mathrm{lisse\, et}\,\omega_{X/S}\,f-\mathrm{ample}}
\end{array}\right.
$$
où $\mathcal{S}ch$ désigne la catégorie des schémas sur $\CC$ et $\mathcal{E}ns$ la catégorie des ensembles. La notation $h$ désigne également un polynôme fixé à coefficients rationnels vérifiant $h(\ZZ)\subset \ZZ$ et il est sous-entendu que les fibres $X_s$ de $f:X\to S$ ont $h$ pour polynôme de Hilbert relativement à la polarisation $\omega_{X_s}$.

Ce foncteur est, d'une certaine manière, représenté par un schéma \cite[th. 1.11]{livreVie}.
\begin{theo}\label{th:espace de modules}
$\mathcal{M}_h$ admet un espace de modules grossier qui est un schéma quasi-projectif $M_h$ et qui vérifie donc :
\begin{enumerate}
\item les points complexes $M_h(\CC)$ sont en bijection avec $\mathcal{M}_h(\mathrm{Spec}(\CC))$ l'ensemble des classes d'isomorphismes de variétés canoniquement polarisées.
\item il existe une transformation naturelle $\mathfrak{m}:\mathcal{M}_h\to \mathrm{Hom}(\bullet,M_h)$ : toute famille $f:X\to S$ induit donc un unique morphisme (application modulaire) $\mathfrak{m}_f:S\to M_h$.
\item la transformation $\mathfrak{m}$ est universelle : toute transformation $\mathcal{M}_h\to \mathrm{Hom}(\bullet,Z)$ (avec $Z$ un schéma) se factorise à travers $\mathfrak{m}$.
\end{enumerate}
\end{theo}
Si $f:X\to S$ est une famille de variétés canoniquement polarisées, quitte à restreindre la base nous supposerons que $S$ (et donc $X$) est lisse et nous noterons :
$$\mathrm{Var}(f):=\dim(\mathfrak{m}_f(S)).$$
Dans \cite{V83}, la variation birationnelle est définie pour toute fibration entre variétés algébriques. Dans le cas précis, elle se trouve donnée par la formule ci-dessus et elle s'obtient également en considérant le rang générique de l'application de Kodaira-Spencer $T_{S}\to R^1f_*T_{X/S}$.

\subsubsection{Base des familles de variétés canoniquement polarisées}

Les résultats obtenus dans la deuxième partie de ce texte vont permettre de décrire en quelque sorte l'image de l'application $\mathfrak{m}_f$ associée à une famille de variétés canoniquement polarisées. L'hyperbolicité algébrique auquel le titre fait référence revient dans un sens vague à dire que les sous-variétés de $M_h$ sont de log-type général. Plus précisément, nous établissons l'énoncé suivant, démontré dans \cite{Taj}.
\begin{theo}\label{th:factorisation modulaire/coeur}
Soit $f:X^\circ\to Y^\circ$ une famille de variétés canoniquement polarisées ($f$ est un morphisme lisse entre variétés quasi-projectives lisses) et considérons l'application induite vers l'espace des modules :
$$\mathfrak{m}_f:Y^\circ\to \mathcal{M}_h.$$
Cette application factorise par le c{\oe}ur de $Y^0$ comme défini\footnote{Le c{\oe}ur de $Y^\circ$ est par définition le c{\oe}ur de l'orbifolde $(Y,D)$ avec $Y$ une compactification lisse et $D:=Y\setminus Y^\circ$ à croisements normaux. De même, nous dirons que $Y^\circ$ est spéciale si $(Y,D)$ l'est.} plus haut :
$$\xymatrix{Y^\circ\ar[rr]^{\mathfrak{m}_f}\ar[rd]_{c_{Y^\circ}} &&\mathcal{M}_h \\
&C(Y^\circ).\ar[ru]&
}$$
En particulier,
\begin{enumerate}
\item si la variation de $f$ est maximale $(\dim(\mathfrak{m}_f(Y^\circ))=\dim(Y^\circ))$, $Y^\circ$ est de log-type général (conjecture de Viehweg).
\item si $Y^\circ$ est spéciale, $f$ est isotriviale (conjecture de Campana).
\end{enumerate}
\end{theo}
Il faut noter ici que le point 1 est établi dans \cite{CP13}. L'ingrédient essentiel pour montrer cette factorisation est le résultat suivant de Viehweg et Zuo \cite{VZ00} qui stipule que la base d'une famille de variétés canoniquement polarisées porte beaucoup de différentielles symétriques.
\begin{theo}\label{th:VZ}
Si $f:X^\circ\to Y^\circ$ est une famille de variété canoniquement polarisées et si $Y$ désigne une compactification de $Y^\circ$ avec $D:=Y\setminus Y^\circ$ un diviseur à croisements normaux. Il existe alors un entier $m\ge1$ et un faisceau inversible $\mathcal{A}\subset \sym^m\Omega^1_Y(\log(D))$ vérifiant de plus : $\kappa(\mathcal{A})\ge\mathrm{Var}(f)$.
\end{theo}
Les travaux de Jabbusch et Kebekus \cite{JKaif} ont permis d'affiner le résultat précédent et montrent que le faisceau $\mathcal{A}$ provient en quelque sorte de l'espace de modules. Considérons pour cela $\mathfrak{m}_f:Y^\circ\to M_h$ l'application de modules et son prolongement (encore noté $\mathfrak{m}_f$) $Y\to \bar{M}_h$, où $\bar{M}_h$ est une compactification projective de $M_h$. Considérons la factorisation de Stein $g:Y\to Z$ de $\mathfrak{m}_h$ : $g$ est donc surjective, à fibres connexes et $\dim(Z)=\mathrm{Var}(f)$. Quitte à modifier $Y$ et $Z$, nous supposerons que $Z$ est lisse et que nous pouvons lui appliquer la conclusion de la proposition \ref{prop:existence base orbifolde} : il existe sur $Z$ une structure orbifolde induite par $g$ et $D$ et que nous noterons $\Delta$. Une fois ce cadre fixé, le résultat \cite[th. 1.4]{JKaif} s'énonce comme suit.
\begin{theo}\label{th:raffinement VZ}
Il existe un faisceau inversible $\mathcal{A}_Z\subset \mathsf{S}^m\Omega^1_Z(\Delta)$ qui est de plus \emph{big} :
$$\kappa(\mathcal{A}_Z)=\mathrm{Var}(f)=\dim(Z).$$
\end{theo}
Fort de ces informations, nous pouvons conclure la
\begin{proof}[Démonstration du théorème \ref{th:factorisation modulaire/coeur}]
Le corollaire \ref{cor:critère type général} montre que la paire $(Z,\Delta)$ est de type général. Comme, d'après la proposition \ref{prop:existence base orbifolde}, nous avons $\kappa(Z,\Delta)=\kappa_D(g)$, la fibration $g:(Y,D)\to Z$ est donc de type général. Il nous suffit alors d'appliquer le théorème \ref{th:decomposition du coeur} : le c{\oe}ur de $(Y,D)$ doit dominer la fibration $g$ et cela fournit la factorisation annoncée.
\end{proof}


\begin{thebibliography}{BCHM10}

\bibitem[Ara71]{A71}
{\scshape S.~J. Arakelov} -- {\og Families of algebraic curves with fixed
  degeneracies\fg}, \emph{Izv. Akad. Nauk SSSR Ser. Mat.} \textbf{35} (1971),
  p.~1269--1293.

\bibitem[BCHM10]{BCHM}
{\scshape C.~Birkar, P.~Cascini, C.~D. Hacon {\normalfont \smfandname}
  J.~McKernan} -- {\og Existence of minimal models for varieties of log general
  type\fg}, \emph{J. Amer. Math. Soc.} \textbf{23} (2010), no.~2, p.~405--468.

\bibitem[BDPP13]{BDPP}
{\scshape S.~Boucksom, J.-P. Demailly, M.~P{\u{a}}un {\normalfont \smfandname}
  T.~Peternell} -- {\og The pseudo-effective cone of a compact {K}\"ahler
  manifold and varieties of negative {K}odaira dimension\fg}, \emph{J.
  Algebraic Geom.} \textbf{22} (2013), no.~2, p.~201--248.

\bibitem[BM01]{BM01}
{\scshape F.~Bogomolov {\normalfont \smfandname} M.~McQuillan} -- {\og Rational
  curves on foliated varieties\fg}, preprint IHES/M/01/07 (2001).

\bibitem[Bog78]{Bog}
{\scshape F.~A. Bogomolov} -- {\og Holomorphic tensors and vector bundles on
  projective manifolds\fg}, \emph{Izv. Akad. Nauk SSSR Ser. Mat.} \textbf{42}
  (1978), no.~6, p.~1227--1287, 1439.

\bibitem[Bos01]{JBB}
{\scshape J.-B. Bost} -- {\og Algebraic leaves of algebraic foliations over
  number fields\fg}, \emph{Publ. Math. Inst. Hautes \'Etudes Sci.} (2001),
  no.~93, p.~161--221.

\bibitem[Cam04]{Ca04}
{\scshape F.~Campana} -- {\og Orbifolds, special varieties and classification
  theory\fg}, \emph{Ann. Inst. Fourier (Grenoble)} \textbf{54} (2004), no.~3,
  p.~499--630.

\bibitem[Cam11]{Ca11j}
\bysame , {\og Orbifoldes g\'eom\'etriques sp\'eciales et classification
  bim\'eromorphe des vari\'et\'es k\"ahl\'eriennes compactes\fg}, \emph{J.
  Inst. Math. Jussieu} \textbf{10} (2011), no.~4, p.~809--934.

\bibitem[CPet11]{CPet}
{\scshape F.~Campana {\normalfont \smfandname} T.~Peternell} -- {\og Geometric
  stability of the cotangent bundle and the universal cover of a projective
  manifold\fg}, \emph{Bull. Soc. Math. France} \textbf{139} (2011), no.~1,
  p.~41--74, With an appendix by Matei Toma.

\bibitem[CP13]{CP13}
{\scshape F.~Campana {\normalfont \smfandname} M.~P{\u a}un} -- {\og Orbifold
  generic semi-positivity: an application to family of canonically polarized
  manifolds\fg}, {\`a} para{\^i}tre aux \emph{Ann. Inst. Fourier (Grenoble)}, arXiv:1303.3169 (2013).

\bibitem[CP14]{CP14}
\bysame , {\og Positivity properties of the bundle of logaithmic tensors on
  compact k{\"a}hler manifolds\fg}, arXiv:1407.3431 (2014).

\bibitem[CP15]{CP15}
\bysame , {\og Foliations with positive slopes and birational stability of
  orbifold cotangent bundles\fg}, arXiv:1508.02456 (2015).

\bibitem[Dru15]{D15}
{\scshape S.~Druel} -- {\og On foliations with nef anti-canonical bundle\fg},
  arXiv:1507.01348 (2015).

\bibitem[Fuj14]{Fuj14}
{\scshape O.~Fujino} -- {\og Notes on the weak positivity theorems\fg}, {\`a} para{\^i}tre dans \emph{Adv. Stud. Pure Math.}, arXiv:1406.1834 (2014).

\bibitem[GHS03]{GHS}
{\scshape T.~Graber, J.~Harris {\normalfont \smfandname} J.~Starr} -- {\og
  Families of rationally connected varieties\fg}, \emph{J. Amer. Math. Soc.}
  \textbf{16} (2003), no.~1, p.~57--67 (electronic).

\bibitem[GKKP11]{GKKP}
{\scshape D.~Greb, S.~Kebekus, S.~J. Kov{\'a}cs {\normalfont \smfandname}
  T.~Peternell} -- {\og Differential forms on log canonical spaces\fg},
  \emph{Publ. Math. Inst. Hautes \'Etudes Sci.} (2011), no.~114, p.~87--169.

\bibitem[GKP15]{GKP}
{\scshape D.~Greb, S.~Kebekus {\normalfont \smfandname} T.~Peternell} -- {\og
  Movable curves and semistable sheaves\fg}, {\`a} para{\^i}tre {\`a} \emph{Int. Math. Res. Not. IMRN}, disponible {\`a} l'adresse \texttt{http://dx.doi.org/10.1093/imrn/rnv126} (2015).

\bibitem[Har68]{H68}
{\scshape R.~Hartshorne} -- {\og Cohomological dimension of algebraic
  varieties\fg}, \emph{Ann. of Math. (2)} \textbf{88} (1968), p.~403--450.

\bibitem[H{\"o}r12]{Hor}
{\scshape A.~H{\"o}ring} -- {\og On a conjecture of {B}eltrametti and
  {S}ommese\fg}, \emph{J. Algebraic Geom.} \textbf{21} (2012), no.~4,
  p.~721--751.

\bibitem[JK11a]{JKmz}
{\scshape K.~Jabbusch {\normalfont \smfandname} S.~Kebekus} -- {\og Families
  over special base manifolds and a conjecture of {C}ampana\fg}, \emph{Math.
  Z.} \textbf{269} (2011), no.~3-4, p.~847--878.

\bibitem[JK11b]{JKaif}
\bysame , {\og Positive sheaves of differentials coming from coarse moduli
  spaces\fg}, \emph{Ann. Inst. Fourier (Grenoble)} \textbf{61} (2011), no.~6,
  p.~2277--2290 (2012).

\bibitem[Kaw81]{K81}
{\scshape Y.~Kawamata} -- {\og Characterization of abelian varieties\fg},
  \emph{Compositio Math.} \textbf{43} (1981), no.~2, p.~253--276.

\bibitem[Keb13]{surveyKeb}
{\scshape S.~Kebekus} -- {\og Differential forms on singular spaces, the
  minimal model program, and hyperbolicity of moduli stacks\fg}, in
  \emph{Handbook of moduli. {V}ol. {II}}, Adv. Lect. Math. (ALM), vol.~25, Int.
  Press, Somerville, MA, 2013, p.~71--113.

\bibitem[KK08]{KK08}
{\scshape S.~Kebekus {\normalfont \smfandname} S.~J. Kov{\'a}cs} -- {\og
  Families of canonically polarized varieties over surfaces\fg}, \emph{Invent.
  Math.} \textbf{172} (2008), no.~3, p.~657--682.

\bibitem[Kob98]{livreKob}
{\scshape S.~Kobayashi} -- \emph{Hyperbolic complex spaces}, Grundlehren der
  Mathematischen Wissenschaften [Fundamental Principles of Mathematical
  Sciences], vol. 318, Springer-Verlag, Berlin, 1998.

\bibitem[KSCT07]{KSCT}
{\scshape S.~Kebekus, L.~Sol{\'a}~Conde {\normalfont \smfandname} M.~Toma} --
  {\og Rationally connected foliations after {B}ogomolov and {M}c{Q}uillan\fg},
  \emph{J. Algebraic Geom.} \textbf{16} (2007), no.~1, p.~65--81.

\bibitem[Laz04]{Laz}
{\scshape R.~Lazarsfeld} -- \emph{Positivity in algebraic geometry. {I}},
  Ergebnisse der Mathematik und ihrer Grenzgebiete. 3. Folge. A Series of
  Modern Surveys in Mathematics [Results in Mathematics and Related Areas. 3rd
  Series. A Series of Modern Surveys in Mathematics], vol.~48, Springer-Verlag,
  Berlin, 2004, Classical setting: line bundles and linear series.

\bibitem[LPT11]{LPT}
{\scshape F.~Loray, J.~V. Pereira {\normalfont \smfandname} F.~Touzet} -- {\og
  Singular foliations with trivial canonical class\fg}, arXiv:1107.1538 (2011).

\bibitem[LY87]{LY}
{\scshape J.~Li {\normalfont \smfandname} S.-T. Yau} -- {\og
  Hermitian-{Y}ang-{M}ills connection on non-{K}\"ahler manifolds\fg}, in
  \emph{Mathematical aspects of string theory ({S}an {D}iego, {C}alif., 1986)},
  Adv. Ser. Math. Phys., vol.~1, World Sci. Publishing, Singapore, 1987,
  p.~560--573.

\bibitem[Miy87]{Mi87}
{\scshape Y.~Miyaoka} -- {\og Deformations of a morphism along a foliation and
  applications\fg}, in \emph{Algebraic geometry, {B}owdoin, 1985 ({B}runswick,
  {M}aine, 1985)}, Proc. Sympos. Pure Math., vol.~46, Amer. Math. Soc.,
  Providence, RI, 1987, p.~245--268.

\bibitem[Miy08]{Mi08}
\bysame , {\og The orbibundle {M}iyaoka-{Y}au-{S}akai inequality and an
  effective {B}ogomolov-{M}c{Q}uillan theorem\fg}, \emph{Publ. Res. Inst. Math.
  Sci.} \textbf{44} (2008), no.~2, p.~403--417.

\bibitem[Par68]{P68}
{\scshape A.~N. Parshin} -- {\og Algebraic curves over function fields.
  {I}\fg}, \emph{Izv. Akad. Nauk SSSR Ser. Mat.} \textbf{32} (1968),
  p.~1191--1219.

\bibitem[Pat12]{Pat12}
{\scshape Z.~Patakfalvi} -- {\og Viehweg's hyperbolicity conjecture is true
  over compact bases\fg}, \emph{Adv. Math.} \textbf{229} (2012), no.~3,
  p.~1640--1642.

\bibitem[Sha63]{Sh63}
{\scshape I.~R. Shafarevich} -- {\og Algebraic number fields\fg}, in
  \emph{Proc. {I}nternat. {C}ongr. {M}athematicians ({S}tockholm, 1962)}, Inst.
  Mittag-Leffler, Djursholm, 1963, p.~163--176.

\bibitem[Taj13]{Taj}
{\scshape B.~Taji} -- {\og The isotriviality of smooth families of canonically
  polarized manifolds over a special quasi-projective base\fg}, {\`a} para{\^i}tre {\`a} \emph{Compositio Math.}, arXiv:1310.5391 (2013).

\bibitem[Vie83]{V83}
{\scshape E.~Viehweg} -- {\og Weak positivity and the additivity of the
  {K}odaira dimension for certain fibre spaces\fg}, in \emph{Algebraic
  varieties and analytic varieties ({T}okyo, 1981)}, Adv. Stud. Pure Math.,
  vol.~1, North-Holland, Amsterdam, 1983, p.~329--353.

\bibitem[Vie95]{livreVie}
\bysame , \emph{Quasi-projective moduli for polarized manifolds}, Ergebnisse
  der Mathematik und ihrer Grenzgebiete (3) [Results in Mathematics and Related
  Areas (3)], vol.~30, Springer-Verlag, Berlin, 1995.

\bibitem[VZ02]{VZ00}
{\scshape E.~Viehweg {\normalfont \smfandname} K.~Zuo} -- {\og Base spaces of
  non-isotrivial families of smooth minimal models\fg}, in \emph{Complex
  geometry ({G}\"ottingen, 2000)}, Springer, Berlin, 2002, p.~279--328.

\end{thebibliography}

\providecommand{\bysame}{\leavevmode ---\ }
\providecommand{\og}{``}
\providecommand{\fg}{''}
\providecommand{\smfandname}{\&}
\providecommand{\smfedsname}{\'eds.}
\providecommand{\smfedname}{\'ed.}
\providecommand{\smfmastersthesisname}{M\'emoire}
\providecommand{\smfphdthesisname}{Th\`ese}

\end{document}